\batchmode
\makeatletter
\def\input@path{{\string"/Users/paranoia/Documents/Research/mypapers/Cubical Convex Ear Decompositions/\string"/}}
\makeatother
\documentclass[12pt]{article}
\usepackage[T1]{fontenc}
\usepackage[latin9]{inputenc}
\usepackage{verbatim}
\usepackage{amsthm}
\usepackage{amsmath}
\usepackage{amssymb}

\makeatletter
\theoremstyle{plain}
\newtheorem{thm}{Theorem}[section]
  \theoremstyle{plain}
  \newtheorem{lem}[thm]{Lemma}
  \theoremstyle{remark}
  \newtheorem{note}[thm]{Note}
  \theoremstyle{plain}
  \newtheorem{cor}[thm]{Corollary}
  \newcounter{casectr}
  \newenvironment{caseenv}
  {\begin{list}{{\itshape\ Case} \arabic{casectr}.}{%
   \setlength{\leftmargin}{\labelwidth}
   \addtolength{\leftmargin}{\parskip}
   \setlength{\itemindent}{\listparindent}
   \setlength{\itemsep}{\medskipamount}
   \setlength{\topsep}{\itemsep}}
   \setcounter{casectr}{0}
   \usecounter{casectr}}
  {\end{list}}
 \theoremstyle{definition}
  \newtheorem{example}[thm]{Example}
  \theoremstyle{plain}
  \newtheorem{prop}[thm]{Proposition}
  \theoremstyle{remark}
  \newtheorem{claim}[thm]{Claim}

  \newlength{\BiblioSpacing}
  \renewenvironment{thebibliography}[1]{%
    \begin{oldthebibliography}{#1}%
      \setlength{\parskip}{\BiblioSpacing}
      \setlength{\itemsep}{\BiblioSpacing}
  }%
  {%
    \end{oldthebibliography}%
  }

\newtheorem{question}{Question}

\usepackage{ae}
\usepackage{aecompl}

\makeatother

\begin{document}
%
\begin{comment}
Group theory
\end{comment}
{}

\global\long\def\normalin{\mathrel{\triangleleft}}

\global\long\def\innormal{\mathrel{\triangleright}}

\global\long\def\semidirect{\mathbin{\rtimes}}

\global\long\def\Stab{\operatorname{Stab}}

%
\begin{comment}
Topology
\end{comment}
{}

\global\long\def\bdry{\partial}

\global\long\def\susp{\operatorname{susp}}

%
\begin{comment}
Poset combinatorics
\end{comment}
{}

\global\long\def\lrprod{\mathop{\check{\prod}}}

\global\long\def\lrtimes{\mathbin{\check{\times}}}

\global\long\def\urtimes{\mathbin{\hat{\times}}}

\global\long\def\urprod{\mathop{\hat{\prod}}}

\global\long\def\subsetdot{\mathrel{\subset\!\!\!\!{\cdot}\,}}

\global\long\def\dotsupset{\mathrel{\supset\!\!\!\!\!\cdot\,\,}}

\global\long\def\precdot{\mathrel{\prec\!\!\!\cdot\,}}

\global\long\def\dotsucc{\mathrel{\cdot\!\!\!\succ}}

\global\long\def\des{\operatorname{des}}

%
\begin{comment}
Lattice
\end{comment}
{}

\global\long\def\modreln{\mathrel{M}}

%
\begin{comment}
Simplicial complex combinatorics
\end{comment}
{}

\global\long\def\link{\operatorname{link}}

\global\long\def\freejoin{\mathbin{\circledast}}

\global\long\def\stellarsd{\operatorname{stellar}}

\global\long\def\conv{\operatorname{conv}}

\global\long\def\disjointunion{\mathbin{\dot{\cup}}}

%
\begin{comment}
My stuff
\end{comment}
{}

\global\long\def\cosetposet{\overline{\mathfrak{C}}}

\global\long\def\cosetlat{\mathfrak{C}}

\global\long\def\ltrr{<_{\mathrm{rr-lex}}}

\global\long\def\gtrr{>_{\mathrm{rr-lex}}}

\title{Cubical Convex Ear Decompositions}

\author{Russ Woodroofe\\
{\small Department of Mathematics}\\
{\small Washington University in St.~Louis}\\
{\small St.~Louis, MO 63130, USA}\\
\texttt{\small russw@math.wustl.edu}}

\date{%%\dateline{Aug 8, 2008}{Mar 28, 2009}\\
{\small \medskip{}
Dedicated to Anders Björner in honor of his 60th birthday.}}
\maketitle
\begin{abstract}
We consider the problem of constructing a convex ear decomposition
for a poset. The usual technique, introduced by Nyman and Swartz,
starts with a $CL$-labeling and uses this to shell the `ears' of
the decomposition. We axiomatize the necessary conditions for this
technique as a {}``$CL$-ced'' or {}``$EL$-ced''. We find an
$EL$-ced of the $d$-divisible partition lattice, and a closely related
convex ear decomposition of the coset lattice of a relatively complemented
finite group. Along the way, we construct new $EL$-labelings of both
lattices. The convex ear decompositions so constructed are formed
by face lattices of hypercubes.

We then proceed to show that if two posets $P_{1}$ and $P_{2}$ have
convex ear decompositions ($CL$-ceds), then their products $P_{1}\times P_{2}$,
$P_{1}\lrtimes P_{2}$, and $P_{1}\urtimes P_{2}$ also have convex
ear decompositions ($CL$-ceds). An interesting special case is: if
$P_{1}$ and $P_{2}$ have polytopal order complexes, then so do their
products. 
\end{abstract}
\tableofcontents{}

\section{Introduction}

Convex ear decompositions, introduced by Chari in \cite{Chari:1997},
break a simplicial complex into subcomplexes of convex polytopes in
a manner with nice properties for enumeration. A complex with a convex
ear decomposition inherits many properties of convex polytopes. For
example, such a complex has a unimodal $h$-vector \cite{Chari:1997},
with an analogue of the $g$-theorem holding \cite{Swartz:2006},
and is doubly Cohen-Macaulay \cite{Swartz:2006}.

Nyman and Swartz constructed a convex ear decomposition for geometric
lattices in \cite{Nyman/Swartz:2004}. Their proof method used the
$EL$-labeling of such lattices to understand the decomposition's
topology. Similar techniques were pushed further by Schweig \cite{Schweig:2006}.
In Section \ref{sec:Definitions-and-Tools}, we introduce the necessary
background material, and then axiomatize the conditions necessary
for these techniques. We call such a convex ear decomposition a {}``$CL$-ced'',
or {}``$EL$-ced.''

We then show by example in Sections \ref{sec:d-divisible-Partition}
and \ref{sec:coset-lattice} how to use these techniques on some poset
families: $d$-divisible partition lattices, and coset lattices of
a relatively complemented group. These posets have each interval $[a,\hat{1}]$
supersolvable, where $a\neq\hat{0}$. Finding the convex ear decompositions
will involve constructing a (dual) $EL$-labeling that respects the
supersolvable structure up to sign, and showing that a set of (barycentricly
subdivided) hypercubes related to the $EL$-labeling is an $EL$-ced,
or at least a convex ear decomposition. We will prove specifically:
\begin{thm}
\label{thm:d-divisibleced_intro}The $d$-divisible partition lattice
$\Pi_{n}^{d}$ has an $EL$-ced, hence a convex ear decomposition.
\end{thm}

\begin{thm}
\label{thm:cosetposetced_intro}The coset lattice $\cosetlat(G)$
has a convex ear decomposition if and only if $G$ is a relatively
complemented finite group.
\end{thm}
I believe these convex ear decompositions to be the first large class
of examples where each ear is a hypercube. 

Although both poset families were known to be $EL$-shellable, the
$EL$-labelings that we construct in these sections also seem to be
new. The ideas used to find them may be applicable in other settings,
as briefly discussed in Section \ref{sec:Further-Questions}.
\begin{lem}
$\Pi_{n}^{d}$ has a dual $EL$-labeling; $\cosetlat(G)$ has a dual
$EL$-labeling if $G$ is a complemented finite group.
\end{lem}
In Section \ref{sec:Poset-Products} we change focus slightly to discuss
products of bounded posets. Our first goal is:
\begin{thm}
\label{thm:cedprod_intro}If bounded posets $P_{1}$ and $P_{2}$
have convex ear decompositions, then so do $P_{1}\times P_{2}$, $P_{1}\lrtimes P_{2}$,
and $P_{1}\urtimes P_{2}$.
\end{thm}
This is the first result of which I am aware that links poset constructions
and convex ear decompositions with such generality. A result of a
similar flavor (but more restrictive) is proved by Schweig \cite{Schweig:2006}:
that rank selected subposets of some specific families of posets have
convex ear decompositions.

A special case of Theorem \ref{thm:cedprod_intro} has a particularly
pleasing form:
\begin{lem}
\label{lem:posetpolytopeprod_intro}If $P_{1}$ and $P_{2}$ are bounded
posets such that $\vert P_{1}\vert$ and $\vert P_{2}\vert$ are isomorphic
to the boundary complexes of simplicial polytopes, then so are $\vert P_{1}\times P_{2}\vert$,
$\vert P_{1}\lrtimes P_{2}\vert$, and $\vert P_{1}\urtimes P_{2}\vert$.
\end{lem}
We then recall the work of Björner and Wachs \cite[Section 10]{Bjorner/Wachs:1997}
on $CL$-labelings of poset products, which we use to prove a result
closely related to Theorem \ref{thm:cedprod_intro}:
\begin{thm}
\label{thm:clcedprod_intro}If bounded posets $P_{1}$ and $P_{2}$
have $CL$-ceds with respect to $CL$-labelings $\lambda_{1}$ and
$\lambda_{2}$, then $P_{1}\times P_{2}$, $P_{1}\lrtimes P_{2}$,
and $P_{1}\urtimes P_{2}$ have $CL$-ceds with respect to the labelings
$\lambda_{1}\times\lambda_{2}$, $\lambda_{1}\lrtimes\lambda_{2}$,
and $\lambda_{1}\urtimes\lambda_{2}$.
\end{thm}
We close by considering some additional questions and directions for
further research in Section \ref{sec:Further-Questions}.

\section{\label{sec:Definitions-and-Tools}Definitions and tools}

All simplicial complexes, posets, and groups discussed in this paper
are finite.

A poset $P$ is \emph{bounded} if it has a lower bound $\hat{0}$
and an upper bound $\hat{1}$, so that $\hat{0}\leq x\leq\hat{1}$
for all $x\in P$. 

If $P$ is a bounded poset, then the \emph{order complex} $\vert P\vert$
is the simplicial complex whose faces are the chains of $P\setminus\{\hat{0},\hat{1}\}$.
(This is slightly different from the standard definition, in that
we are taking only the \emph{proper part} of the poset.) Where it
will cause no confusion, we talk about $P$ and $\vert P\vert$ interchangeably:
for example, we say $P$ has a convex ear decomposition if $\vert P\vert$
does. 

We denote by $\mathcal{M}(P)$ the set of maximal chains of $P$,
which is in natural bijective correspondence with the facets of $\vert P\vert$
through adding or removing $\hat{0}$ and $\hat{1}$.

\subsection{Convex ear decompositions}

A \emph{convex ear decomposition} of a pure $(d-1)$-dimensional simplicial
complex $\Delta$ is an ordered collection of subcomplexes $\Delta_{1},\dots,\Delta_{m}\subseteq\Delta$
with the following properties:
\begin{description}
\item [{ced-polytope}] $\Delta_{s}$ is isomorphic to a subcomplex of the
boundary complex of a simplicial $d$-polytope for each $s$. 
\item [{ced-topology}] $\Delta_{1}$ is a $(d-1)$-sphere, and $\Delta_{s}$
is a $(d-1)$-ball for $s>1$.
\item [{ced-bdry}] $(\bigcup_{t=1}^{s-1}\Delta_{t})\cap\Delta_{s}=\bdry\Delta_{s}$
for each $s>1$.
\item [{ced-union}] $\bigcup_{s=1}^{m}\Delta_{s}=\Delta$.
\end{description}
It follows immediately from the definition that any complex with a
convex ear decomposition is pure. As far as I know, no one has tried
generalizing the theory of convex ear decompositions to non-pure complexes.
As many interesting posets are not graded (i.e., have an order complex
that is not pure), finding such a generalization could be useful.

Convex ear decompositions were first introduced by Chari \cite{Chari:1997}.
He used the unimodality of the $h$-vector of a simplicial polytope
to give a strong condition on the $h$-vector for a complex with a
convex ear decomposition. Swartz \cite{Swartz:2006} showed that a
`$g$-theorem' holds for any $(d-1)$-dimensional complex with a convex
ear decomposition, as stated precisely in Theorem \ref{thm:HvectorCEDSwartz}.
We refer the reader to \cite{Stanley:1996} for further background
on $h$-vectors, $M$-vectors, and the (original) $g$-theorem.
\begin{thm}
\emph{\label{thm:HvectorCEDChari}(Chari \cite[Section 3]{Chari:1997})}
The $h$-vector of a pure $(d-1)$-dimensional complex with a convex
ear decomposition satisfies the conditions \[
h_{0}\leq h_{1}\leq\dots\leq h_{\lfloor d/2\rfloor}\]
\[
h_{i}\leq h_{d-i},\quad\mbox{for }0\leq i\leq\lfloor d/2\rfloor.\]

\end{thm}

\begin{thm}
\emph{\label{thm:HvectorCEDSwartz}(Swartz \cite[Corollary 3.10]{Swartz:2006})}
If $\{h_{i}\}$ is the $h$-vector of a pure $(d-1)$-dimensional
complex with a convex ear decomposition, then \[
(h_{0},h_{1}-h_{0},\dots,h_{\lfloor d/2\rfloor}-h_{\lfloor d/2\rfloor-1})\]
is an $M$-vector.
\end{thm}

\subsection{\label{sub:Shellings}Shellings}

An essential tool for us will be the theory of lexicographic shellability,
developed by Björner and Wachs in \cite{Bjorner:1980,Bjorner/Wachs:1983,Bjorner/Wachs:1996,Bjorner/Wachs:1997}.
We recall some of the main facts.

We say that an ordering of the facets $F_{1},F_{2},\dots,F_{t}$ of
a simplicial complex $\Delta$ (with $t$ facets) is a \emph{shelling}
if $F_{i}\cap\left(\bigcup_{j=1}^{i-1}F_{j}\right)$ is pure $(\dim F_{i}-1)$-dimensional
for all $1<i\leq t$. An equivalent condition that is often easier
to use is:\begin{eqnarray}
\mbox{if }1\leq i<j\leq t,\!\!\!\!\!\! &  & \mbox{then }\exists k<j\mbox{ such that}\label{eq:ShellingDef}\\
 &  & F_{i}\cap F_{j}\subseteq F_{k}\cap F_{j}=F_{j}\setminus\{x\}\mbox{ for some }x\in F_{j}.\nonumber \end{eqnarray}
A simplicial complex is \emph{shellable} if it has a shelling.

The existence of a shelling tells us a great deal about the topology
of a pure $d$-dimensional complex: the complex is Cohen-Macaulay,
with homotopy type a bouquet of spheres of dimension $d$. A fact
about shellable complexes that will be especially useful for us is
that a shellable proper pure $d$-dimensional subcomplex of a simplicial
$d$-sphere is a $d$-ball \cite[Proposition 1.2]{Danaraj/Klee:1974}.

A \emph{cover relation} in a poset $P$, denoted $x\lessdot y$, is
a pair $x\lneq y$ of elements in $P$ such that there is no $z\in P$
with $x\lneq z\lneq y$. Equivalently, a cover relation is an edge
in the Hasse diagram of $P$.

An \emph{$EL$-labeling} of $P$ (where $EL$ stands for \emph{edge
lexicographic}) is a map from the cover relations of $P$ to some
fixed partially ordered set, such that in any interval $[x,y]$ there
is a unique increasing maximal chain (i.e., a unique chain with increasing
labels, read from the bottom), and this chain is lexicographically
first among maximal chains in $[x,y]$. It is a well-known theorem
of Björner in the pure case \cite[Theorem 2.3]{Bjorner:1980}, and
more generally of Björner and Wachs \cite[Theorem 5.8]{Bjorner/Wachs:1996},
that any bounded poset $P$ with an $EL$-labeling is shellable. As
a result, the term \emph{$EL$-shelling} is sometimes used as a synonym
of $EL$-labeling.

The families of posets that we study in this paper will have lower
intervals $[\hat{0},x]$ that `look like' the whole poset, but upper
intervals $[x,\hat{1}]$ of a different form. For induction, then,
it will usually be easier for us to label the posets upside down,
and construct \emph{dual $EL$-labelings}, that is, $EL$-labelings
of the dual poset. Dual $EL$-labelings have been used in other settings,
and are natural in many contexts \cite[Corollary 4.4]{Bjorner/Wachs:1983}
\cite[Corollary 4.10]{Shareshian:2001}.

A generalization of an $EL$-labeling which is sometimes easier to
construct (though harder to think about) is that of a $CL$-labeling.
Here, instead of labeling the cover relations (edges), we label {}``rooted
edges.'' More precisely, a \emph{rooted edge}, or \emph{rooted cover
relation,} is a pair $(\mathbf{r},x\lessdot y)$, where the \emph{root}
$\mathbf{r}$ is any maximal chain from $\hat{0}$ to $x$. Also,
if $x_{0}\lessdot x_{1}\lessdot\dots\lessdot x_{n}$ is a maximal
chain on $[x_{0},x_{n}]$, and $\mathbf{r}$ is a root for $x_{0}\lessdot x_{1}$,
then $\mathbf{r}\cup\{x_{1}\}$ is a root for $x_{1}\lessdot x_{2}$,
and so on, so it makes sense to talk of a \emph{rooted chain} $\mathbf{c}_{\mathbf{r}}$
on a \emph{rooted interval} $[x_{0},x_{n}]_{\mathbf{r}}$. A \emph{$CL$-labeling}
is one where every rooted interval $[x,z]_{\mathbf{r}}$ has a unique
increasing maximal chain, and the increasing chain is lexicographically
first among all chains in $[x,z]_{\mathbf{r}}$. An in-depth discussion
of $CL$-labelings can be found in \cite{Bjorner/Wachs:1983,Bjorner/Wachs:1996}:
the main fact is that $EL$-shellable$\implies$$CL$-shellable$\implies$shellable.
We will make real use of the greater generality of $CL$-labelings
only in Section \ref{sec:Poset-Products}, and the unfamiliar reader
is encouraged to read {}``$EL$'' for {}``$CL$'' everywhere else.

The homotopy type of bounded posets with a $CL$-labeling (including
an $EL$-labeling) is especially easy to understand, as discussed
in \cite{Bjorner/Wachs:1996}. Such a poset is homotopy equivalent
to a bouquet of spheres, with the spheres in one-to-one correspondence
with the descending maximal chains. These descending chains moreover
form a cohomology basis for $\vert P\vert$.

\subsection{Supersolvable lattices\label{sub:Supersolvable-lattices}}

The upper intervals $[x,\hat{1}]$ in the posets we look at will be
supersolvable, so we mention some facts about supersolvable lattices.
For additional background, the reader is referred to \cite{Stanley:1972}
or \cite{McNamara/Thomas:2006}.

An element $x$ of a lattice $L$ is \emph{left modular} if for every
$y\leq z$ in $L$ it holds that $(y\vee x)\wedge z=y\vee(x\wedge z)$.
This looks a great deal like the well-known Dedekind identity from
group theory, and in particular any normal subgroup is left modular
in the subgroup lattice.

A graded lattice is \emph{supersolvable} if there is a maximal chain
$\hat{1}=x_{0}\gtrdot x_{1}\gtrdot\dots\gtrdot x_{d}=\hat{0}$, where
each $x_{i}$ is left modular. Thus the subgroup lattice of a supersolvable
group is a supersolvable lattice. In fact, supersolvable lattices
were introduced to generalize the lattice properties of supersolvable
groups.

A supersolvable lattice has a dual $EL$-labeling \[
\lambda^{ss}(y\gtrdot z)=\min\{j\,:\, x_{j}\wedge y\leq z\}=\max\{j-1\,:\, x_{j}\vee z\geq y\},\]
 which we call the \emph{supersolvable labeling} of $L$ (relative
to the given chain of left modular elements). This labeling has the
property:

\begin{equation}
\begin{array}{l}
\mbox{Given an interval }[x,y]\mbox{, every chain on }[x,y]\mbox{ has the same set}\\
\mbox{\quad of labels (in different orders). }\end{array}\label{eq:SSPermutationProperty}\end{equation}
McNamara \cite{McNamara:2003} has shown that having an $EL$-labeling
that satisfies (\ref{eq:SSPermutationProperty}) characterizes the
supersolvable lattices.

\subsection{Cohen-Macaulay complexes}

If $F$ is a face in a simplicial complex $\Delta$, then the link
of $F$ in $\Delta$ is \[
\link_{\Delta}F=\{G\in\Delta\,:\, G\cap F=\emptyset\mbox{ and }G\cup F\in\Delta\}.\]
 A simplicial complex $\Delta$ is \emph{Cohen-Macaulay} if the link
of every face has the homology of a bouquet of top dimensional spheres,
that is, if $\tilde{H}_{i}(\link_{\Delta}F)=0$ for all $i<\dim(\link_{\Delta}F)$. 

The Cohen-Macaulay property has a particularly nice formulation on
the order complex of a poset. A poset is Cohen-Macaulay if every interval
$[x,y]$ has $\tilde{H}_{i}([x,y])=0$ for all $i<\dim(\vert[x,y]\vert)$.
In particular, every interval in a Cohen-Macaulay poset is Cohen-Macaulay.
It is well-known that every shellable complex is Cohen-Macaulay. For
a proof of this fact and additional background on Cohen-Macaulay complexes
and posets, see \cite{Stanley:1996}.

The Cohen-Macaulay property is essentially a connectivity property.
Just as we say a graph $G$ is \emph{doubly connected} (or 2-connected)
if $G$ is connected and $G\setminus\{v\}$ is connected for each
$v\in G$, we say that a simplicial complex $\Delta$ is \emph{doubly
Cohen-Macaulay} (2-CM) if
\begin{enumerate}
\item $\Delta$ is Cohen-Macaulay, and
\item for each vertex $x\in\Delta$, the induced complex $\Delta\setminus\{x\}$
is Cohen-Macaulay of the same dimension as $\Delta$.
\end{enumerate}
Doubly Cohen-Macaulay complexes are closely related to complexes with
convex ear decompositions:
\begin{thm}
\label{thm:CEDthen2CM}\emph{(Swartz \cite{Swartz:2006})} If $\Delta$
has a convex ear decomposition, then $\Delta$ is doubly Cohen-Macaulay.
\end{thm}
Thus, convex ear decompositions can be thought of as occupying an
analogous role to shellings in the geometry of simplicial complexes:
a shelling is a combinatorial reason for a complex to be (homotopy)
Cohen-Macaulay, and a convex ear decomposition is a combinatorial
reason for a complex to be doubly Cohen-Macaulay. Of course, convex
ear decompositions also give the strong enumerative constraints of
Theorems \ref{thm:HvectorCEDChari} and \ref{thm:HvectorCEDSwartz}. 

Intervals in a poset with a convex ear decomposition are not known
to have convex ear decompositions. However, intervals do inherit the
2-CM property, as intervals are links in the order complex, and intervals
inherit the Cohen-Macalay property. Thus, Theorem \ref{thm:CEDthen2CM}
is particularly useful in proving that a poset does not have a convex
ear decomposition.

\subsection{\label{sub:ELcedsAndCLceds}$EL$-ceds and $CL$-ceds}

Nyman and Swartz used an $EL$-labeling in \cite{Nyman/Swartz:2004}
to find a convex ear decomposition for any geometric lattice. The
condition on an $EL$-labeling says that ascending chains are unique
in every interval, and that the lexicographic order of maximal chains
is a shelling. Starting with the usual $EL$-labeling of a geometric
lattice, Nyman and Swartz showed that descending chains are unique
in intervals of an ear of their decomposition, and that the reverse
of the lexicographic order is a shelling. Schweig used similar techniques
in \cite{Schweig:2006} to find convex ear decompositions for several
families of posets, including supersolvable lattices with complemented
intervals.

In this subsection, we axiomatize the conditions necessary for these
techniques. Although we state everything in terms of $CL$-labelings,
one could just as easily read `$EL$' for the purposes of this section,
and ignore the word `rooted' whenever it occurs.

\smallskip{}

Suppose that $P$ is a bounded poset of rank $k$. Let $\{\Sigma_{s}\}$
be an ordered collection of rank $k$ subposets of $P$. For each
$s$, let $\Delta_{s}$ be the simplicial subcomplex generated by
all maximal chains that occur in $\Sigma_{s}$, but not in any $\Sigma_{t}$
for $t<s$. (Informally, $\Delta_{s}$ is all {}``new'' maximal
chains in $\Sigma_{s}$.) Recall that $\mathcal{M}(\Sigma_{s})$ refers
to the maximal chains of $\Sigma_{s}$, and let $\mathcal{M}(\Delta_{s})$
be the maximal chains of $\Delta_{s}$. As usual, maximal chains are
in bijective correspondence with facets of the order complex via removing
or adding $\hat{1}$ and $\hat{0}$.

The ordered collection $\{\Sigma_{s}\}$ is a \emph{chain lexicographic
convex ear decomposition} (or \emph{$CL$-ced} for short) of $P$
with respect to the $CL$-labeling $\lambda$, if it obeys the following
properties:
\begin{description}
\item [{CLced-polytope}] For each $s$, $\Sigma_{s}$ is the face lattice
of a convex polytope.
\item [{CLced-desc}] For any $\Delta_{s}$ and rooted interval $[x,y]_{\mathbf{r}}$
in $P$, there is at most one descending maximal chain $\mathbf{c}$
on $[x,y]_{\mathbf{r}}$ which is a face of $\Delta_{s}$.
\item [{CLced-bdry}] If $\mathbf{c}$ is a chain of length $<k$, such
that $\mathbf{c}$ can be extended to a maximal chain in both of $\Delta_{s}$
and $\Delta_{t}$, where $t<s$; then $\mathbf{c}$ can be extended
to a chain in $\mathcal{M}(\Sigma_{s})\setminus\mathcal{M}(\Delta_{s})$.
\item [{CLced-union}] Every chain in $P$ is in some $\Sigma_{s}$.\end{description}
\begin{note}
\label{not:LikeCLlabelingOnSimpComplex}We note the resemblance of
(CLced-desc) with the increasing chain condition for a $CL$-labeling
(under the reverse ordering of labels); but though $\Delta_{s}$ is
a simplicial complex corresponding with chains in $P$, it is not
itself a poset. 
\end{note}

\begin{note}
\label{not:DescChainIsLexLast}By analogy with $CL$-labelings, it
would seem that we should require the descending chain in (CLced-desc)
to be lexicographically last. But this would be redundant: suppose
$\mathbf{c}$ is the lexicographically last maximal chain in $[x,y]_{\mathbf{r}}$
that is also in $\Delta_{s}$, but that $\mathbf{c}$ has an ascent
at $c_{i}$. Then Lemma \ref{lem:CLcedReplaceAscentWithDescent} below
gives that we can replace the ascent with a descent, obtaining a lexicographically
later chain, a contradiction.
\end{note}

\begin{note}
As previously mentioned, we will usually refer to $EL$-ceds in this
paper, i.e., the special case where $\lambda$ is an $EL$-labeling.
Similarly, we may refer to dual $EL$-ceds, that is, $EL$-ceds of
the dual poset.\end{note}
\begin{lem}
\label{lem:CLcedReplaceAscentWithDescent}\emph{(Technical Lemma)}
Let $\{\Sigma_{s}\}$ be a $CL$-ced of a poset, with $\{\Delta_{s}\}$
as above, and let $\mathbf{c}=\{x\lessdot c_{1}\lessdot\dots\lessdot c_{j-1}\lessdot y\}$
be a maximal chain on a rooted interval $[x,y]_{\mathbf{r}}$, with
$\mathbf{c}$ a face in $\Delta_{s}$. Suppose that $\mathbf{c}$
has an ascent at $c_{i}$. Then $\Delta_{s}$ contains a $\mathbf{c}''=(\mathbf{c}\setminus\{c_{i}\})\cup c_{i}''$
which descends at $c_{i}''$, and is lexicographically later than
$\mathbf{c}$.\end{lem}
\begin{proof}
Let $\mathbf{c}^{-}=\mathbf{c}\setminus\{c_{i}\}$, and let $\Sigma_{t}$
be the first subposet in the $CL$-ced that contains $\mathbf{c}^{-}$.
Since $\Sigma_{t}$ is the face lattice of a polytope, it is Eulerian,
so $\mathbf{c}^{-}$ has two extensions in $\Sigma_{t}$. By the uniqueness
of ascending chains in $CL$-labelings, at most one is ascending at
rank $i$; by (CLced-desc), at most one is descending. Thus, there
is exactly one of each. The extension with the ascent is $\mathbf{c}$,
call the other extension $\mathbf{c}''$.

We have shown that $\mathbf{c}$ is in $\Sigma_{t}$ and (since $\Sigma_{t}$
is the first subposet containing $\mathbf{c}^{-}$) that $s=t$, so
that $\mathbf{c}''$ is in $\Delta_{s}$. Finally, $\mathbf{c}''$
is lexicographically later than $\mathbf{c}$ by the definition of
$CL$-labeling.
\end{proof}
We also recall a useful lemma from undergraduate point-set topology
\cite[Exercise 17.19]{Munkres:2000}:
\begin{lem}
\label{lem:SubmanifoldBdry}If $B$ is a closed subset of $X$, then
$\bdry B=B\cap\overline{X\setminus B}$.
\end{lem}
Although they did not use the terms {}``$CL$-ced'' or {}``$EL$-ced''
in their paper, the essence of the following theorem was proved by
Nyman and Swartz in \cite[Section 4]{Nyman/Swartz:2004}, where they
used it to construct convex ear decompositions of geometric lattices.
\begin{thm}
\label{thm:ELcedsAreCeds}If $\{\Sigma_{s}\}$ is an $CL$-ced for
$P$, then the associated subcomplexes $\{\Delta_{s}\}$ form a convex
ear decomposition for $\vert P\vert$.\end{thm}
\begin{proof}
\emph{(Nyman and Swartz }\cite[Section 4]{Nyman/Swartz:2004}\emph{)}
The property (ced-union) follows directly from (CLced-union), and
(ced-polytope) follows from (CLced-polytope) because the barycentric
subdivision of a polytope is again a polytope. 

For (ced-bdry), we first note that $\bdry\Delta_{s}=\bdry\left(\overline{\vert\Sigma_{s}\vert\setminus\Delta_{s}}\right)$
(the topological closure), hence $\bdry\Delta_{s}\subseteq\Delta_{s}\cap(\bigcup_{t<s}\Delta_{t})$.
Conversely, if $\mathbf{c}$ is in $\Delta_{s}\cap(\bigcup_{t<s}\Delta_{t})$,
then (CLced-bdry) gives that $\mathbf{c}$ is in both $\Delta_{s}$
and $\overline{\vert\Sigma_{s}\vert\setminus\Delta_{s}}$. Lemma \ref{lem:SubmanifoldBdry}
then gives the desired inclusion.

It remains to check (ced-topology). Using (CLced-desc), we show that
the reverse of the lexicographic order is a shelling of $\Delta_{s}$.
For if \begin{align*}
\mathbf{c}\, & =\{\hat{0}\lessdot c_{1}\lessdot\dots\lessdot c_{k-1}\lessdot\hat{1}\}\quad\mbox{ and }\\
\mathbf{c}' & =\{\hat{0}\lessdot c'_{1}\lessdot\dots\lessdot c'_{k-1}\lessdot\hat{1}\}\end{align*}
are maximal chains in $\Delta_{s}$, with $\mathbf{c}$ lexicographically
earlier than $\mathbf{c}'$, then (CLced-desc) and Note \ref{not:DescChainIsLexLast}
give that $\mathbf{c}$ has an ascent on some interval where $\mathbf{c}$
disagrees with $\mathbf{c}'$. So $\mathbf{c}$ has an ascent at $i$,
and $c_{i}\neq c'_{i}$. 

Apply Lemma \ref{lem:CLcedReplaceAscentWithDescent} on the interval
$[\hat{0},\hat{1}]$ to get $\mathbf{c}''$ in $\Delta_{s}$ which
descends at $i$, and otherwise is the same as $\mathbf{c}$. Then
$\mathbf{c'}\cap\mathbf{c}\subseteq\mathbf{c}''\cap\mathbf{c}=\mathbf{c}\setminus\{c_{i}\}$,
so $\vert\mathbf{c}''\cap\mathbf{c}\vert=\vert\mathbf{c}\vert-1$,
and so $\mathbf{c}''$ is lexicographically later than $\mathbf{c}$,
as Condition (\ref{eq:ShellingDef}) requires for a shelling.

We now check that $\Delta_{s}$ is a proper subcomplex of $\vert\Sigma_{s}\vert$
for $s\geq2$. Suppose that $\Delta_{s}=\vert\Sigma_{s}\vert$. Then
by Notes \ref{not:LikeCLlabelingOnSimpComplex} and \ref{not:DescChainIsLexLast},
$\lambda$ is a $CL$-labeling on $\Sigma_{s}$ with respect to the
reverse ordering of its label set. Since $\vert\Sigma_{s}\vert$ is
a sphere, there is an ascending chain (descending chain with respect
to the reverse ordering) in $\Sigma_{s}$. Since the ascending chain
in $P$ is unique, we have $s=1$.

By definition $\Delta_{1}=\vert\Sigma_{1}\vert$ is a $(k-2)$-sphere.
Now since $\Delta_{s}$ is shellable and a proper subcomplex of the
$(k-2)$-sphere $\vert\Sigma_{s}\vert$ for $s\geq2$, we get that
$\Delta_{s}$ is a $(k-2)$-ball; thus (ced-topology) holds.\end{proof}
\begin{note}
Each non-empty ear of $\{\Delta_{s}\}$ contains exactly one descending
chain. This is no accident: see the discussion at the end of Section
\ref{sub:Shellings}.\end{note}
\begin{cor}
\label{thm:ExamplesWCLCED}The following families of posets have $EL$-ceds,
thus convex ear decompositions.
\begin{enumerate}
\item \emph{(Nyman and Swartz \cite[Section 4]{Nyman/Swartz:2004})} Geometric
lattices. 
\item \emph{(Schweig \cite[Theorem 3.2]{Schweig:2006})} Supersolvable lattices
with Möbius function non-zero on every interval.
\item \emph{(Schweig \cite[Theorems 5.1 and 7.1]{Schweig:2006})} Rank-selected
subposets of supersolvable and geometric lattices.
\end{enumerate}
\end{cor}
In the following two sections, we will exhibit an $EL$-ced for the
$d$-divisible partition lattice, and (using only slightly different
techniques) a convex ear decomposition for the coset lattice of a
relatively complemented group.

\section{\label{sec:d-divisible-Partition}The $d$-divisible partition lattice}

The \emph{$d$-divisible partition poset}, denoted $\overline{\Pi}_{n}^{d}$,
is the set of all proper partitions of $[n]=\{1,\dots,n\}$ where
each block has cardinality divisible by $d$. The \emph{$d$-divisible
partition lattice}, denoted $\Pi_{n}^{d}$ is $\overline{\Pi}_{n}^{d}$
with a `top' $\hat{1}$ and `bottom' $\hat{0}$ adjoined. $\Pi_{n}^{d}$
is ordered by refinement (which we denote by $\prec$), as in the
usual \emph{partition lattice} $\Pi_{n}$ ($=\Pi_{n}^{1}$). In general,
$\Pi_{n}^{d}$ is a subposet of $\Pi_{n}$, with equality in the case
$d=1$; on the other hand, intervals $[a,\hat{1}]$ are isomorphic
to $\Pi_{n/d}$ for any atom $a\in\Pi_{n}^{d}$. We refer frequently
to \cite{Wachs:1996} for information about the $d$-divisible partition
lattice. 

As $\Pi_{n}$ is a supersolvable geometric lattice, and hence quite
well understood, we restrict ourself to the case $d>1$. It will sometimes
be convenient to partition a different set $S\neq[n]$. In this case
we write $\Pi_{S}$ to be the set of all partitions of $S$, and $\Pi_{S}^{d}$
the set of all $d$-divisible partitions of $S$, so that $\Pi_{n}^{d}=\Pi_{[n]}^{d}$
is a special case. 

\medskip{}
Wachs found a homology basis for $\Pi_{n}^{d}$ in \cite[Section 2]{Wachs:1996}.
We recall her construction. By $S_{n}$ we denote the symmetric group
on $n$ letters. We will write a permutation $\alpha\in S_{n}$ as
a word $\alpha(1)\alpha(2)\dots\alpha(n)$, and define the \emph{descent
set} of $\alpha$ to be the indices where $\alpha$ descends, i.e.,
$\des\alpha=\{i\in[n-1]\,:\,\alpha(i)>\alpha(i+1)\}$.

Then a \emph{split} of $\alpha\in S_{n}$ at $di$ divides $\alpha$
into $\alpha(1)\alpha(2)\dots\alpha(di)$ and $\alpha(di+1)\dots\alpha(n)$.

A \emph{switch-and-split} at position $di$ does the same, but first
transposes (`switches') $\alpha(di)$ and $\alpha(di+1)$.

These operations can be repeated, and the result of repeated applications
of splits and switch-and-splits at $d$-divisible positions is a $d$-divisible
partition. For example, if $\alpha=561234$, then the $2$-divisible
partition $56\,\vert\,13\,\vert\,24$ results from splitting at position
2 and switch-and-splitting at position 4. 

Let $\Sigma_{\alpha}$ be the subposet of $\Pi_{n}^{d}$ that consists
of all partitions that are obtained by splitting and/or switch-and-splitting
the permutation $\alpha$ at positions divisible by $d$. Let \[
A_{n}^{d}=\left\{ \alpha\in S_{n}\,:\,\alpha(n)=n,\des\alpha=\{d,2d,\dots,n-d\}\right\} .\]
Wachs proved
\begin{thm}
\emph{(Wachs \cite[Theorems 2.1-2.2]{Wachs:1996})}\label{thm:d-divisible-homology-basis}
\begin{enumerate}
\item $\Sigma_{\alpha}$ is isomorphic to the face lattice of the $(\frac{n}{d}-1)$-cube
for any $\alpha\in S_{n}$.
\item $\{\Sigma_{\alpha}\,:\,\alpha\in A_{n}^{d}\}$ is a basis for $H_{*}(\Pi_{n}^{d})$.
\end{enumerate}
\end{thm}
After some work, this basis will prove to be a dual $EL$-ced.

\subsection{A dual $EL$-labeling for $\Pi_{n}^{d}$ \label{sub:d-divisible-EL-labeling}}

In addition to the homology basis already mentioned, Wachs constructs
an $EL$-labeling in \cite[Section 5]{Wachs:1996}, by taking something
close to the standard $EL$-labeling of the geometric lattice on intervals
$[a,\hat{1}]\cong\Pi_{n/d}$ (for $a$ an atom), and {}``twisting''
by making selected labels negative. While her labeling is not convenient
for our purposes, we use her sign idea to construct our own dual $EL$-labeling
starting with a supersolvable $EL$-labeling of $[a,\hat{1}]$.

\medskip{}

Partition lattices were one of the first examples of supersolvable
lattices to be studied \cite{Stanley:1972}. It is not difficult to
see that the maximal chain with $j$th ranked element \[
1\,\vert\,2\,\vert\,\dots\,\vert\, j\,\vert\,(j+1)\dots\, n\]
 is a left modular chain in $\Pi_{n}$.

Let $y\dotsucc z$ be a cover relation in $\Pi_{n}$. Then $y$ is
obtained by merging two blocks $B_{1}$ and $B_{2}$ of the partition
$z$, where without loss of generality $\max B_{1}<\max B_{2}$. The
supersolvable dual $EL$-labeling (relative to the above chain of
left modular elements) is especially natural: \begin{eqnarray*}
\lambda^{ss}(y\dotsucc z) & = & \min\{j\,:\,(1\,\vert\,\dots\,\vert\, j\,\vert\,(j+1)\dots\, n)\wedge y\prec z\}\\
 & = & \max B_{1}.\end{eqnarray*}

We now construct the labeling that we will use for $\Pi_{n}^{d}$.
Let $y\dotsucc z$ be a cover relation in $\Pi_{n}^{d}$, where $z\neq\hat{0}$.
As above, $y$ is obtained by merging blocks $B_{1}$ and $B_{2}$
of $z$, where $\max B_{1}<\max B_{2}$. Label\begin{eqnarray*}
\lambda(y\dotsucc z) & = & \begin{cases}
-\!\!\!\!\! & \max B_{1}\quad\quad\mbox{ if }\max B_{1}<\min B_{2},\\
 & \max B_{1}\quad\quad\mbox{ otherwise, and}\end{cases}\\
\lambda(y\dotsucc\hat{0}) & = & 0.\end{eqnarray*}
When discussing dual $EL$-labelings, any reference to ascending or
descending chains is in the dual poset, so that the inequalities go
in the opposite direction from normal.
\begin{note}
Let $a\in\Pi_{n}^{d}$ be an atom. Then $a$ has $n/d$ blocks, and
every block has $d$ elements. Order the blocks $\{B_{i}\}$ so that
$\max B_{1}<\max B_{2}<\dots<\max B_{n/d}$, and let $\mathcal{B}=\{\max B_{1},\dots,\max B_{n/d}\}$.
Then $[a,1]\cong\Pi_{\mathcal{B}}$, and we recognize $\vert\lambda\vert$
as the supersolvable dual $EL$-labeling $\lambda^{ss}$ on $\Pi_{\mathcal{B}}$.
\end{note}

\begin{note}
We also can view $\Pi_{n}^{d}$ as a subposet of $\Pi_{n}$. A cover
relation $y\dotsucc z$ in $\Pi_{n}^{d}$ is a cover relation in $\Pi_{n}$
unless $z=\hat{0}$. Thus, $\vert\lambda\vert$ is the restriction
of $\lambda^{ss}$ on $\Pi_{n}$, except at the bottom edges $y\dotsucc\hat{0}$.
\end{note}

\begin{note}
The cover relation $x_{0}\dotsucc x_{1}$ gets a negative label if
and only if $B_{1}\,\vert\, B_{2}$ is a \emph{non-crossing partition}
of $B=B_{1}\cup B_{2}$. We will call this a \emph{non-crossing refinement}
of $x_{0}$. The poset of all non-crossing partitions has been studied
extensively \cite{Simion:2000,McCammond:2006}, although this seems
to have a different flavor from what we are doing. Also related is
the \emph{connectivity set} of a permutation \cite{Stanley:2005},
the set of positions at which a split yields a non-crossing partition.
\end{note}
Recall that if $P_{1}$ and $P_{2}$ are posets, then their \emph{direct
product} $P_{1}\times P_{2}$ is the Cartesian product with the ordering
$(x_{1},x_{2})\leq(y_{1},y_{2})$ if $x_{1}\leq x_{2}$ and $y_{1}\leq y_{2}$.
The \emph{lower reduced product} $P_{1}\lrtimes P_{2}$ of two bounded
posets is $\left((P_{1}\setminus\{\hat{0}\})\times(P_{2}\setminus\{\hat{0}\})\right)\cup\{\hat{0}\}$.
Although the definition of the lower reduced product may appear strange
at first glance, it occurs naturally in many settings, including the
following easily-proved lemma:
\begin{lem}
\label{lem:DDivProductStructure}Let $y\succ x$ be elements of $\Pi_{n}^{d}\setminus\{\hat{0}\}$,
with $y=B_{1}\,\vert\,\dots\,\vert\, B_{k}$. Then
\begin{enumerate}
\item $[\hat{0},y]\cong\Pi_{B_{1}}^{d}\lrtimes\Pi_{B_{2}}^{d}\lrtimes\dots\lrtimes\Pi_{B_{k}}^{d}$.
\item $[y,\hat{1}]\cong\Pi_{k}$.
\item $[x,y]$ is the direct product of intervals in $\Pi_{B_{i}}^{d}$.
\end{enumerate}
\end{lem}
\begin{note}
We discuss (lower/upper reduced) products of posets at much more length
in Section \ref{sec:Poset-Products}. Although the situation with
$\Pi_{n}^{d}$ is simple enough that we do not need to refer directly
to product labelings (introduced in Section \ref{sub:ProductCLlabelings}),
they are the underlying reason we can look at partitions block by
block in the proofs that follow. \end{note}
\begin{thm}
$\lambda$ is a dual $EL$-labeling of $\Pi_{n}^{d}$.\end{thm}
\begin{proof}
We need to show that each interval has a unique (dual) increasing
maximal chain which is lexicographically first. There are two forms
of intervals we must check:
\begin{caseenv}
\item Intervals of the form $[\hat{0},x_{0}]$.

Since the bottommost label on every chain in $[\hat{0},x_{0}]$ is
a 0, every other label in an increasing chain must be negative. Hence,
every edge $x_{i}\dotsucc x_{i+1}$ in an increasing chain must correspond
to a non-crossing refinement of $x_{i}$.

In such a chain, any block $B$ of $x_{0}$ is partitioned repeatedly
into non-crossing sub-blocks. At the atom level, this block $B$ is
sub-partitioned as $B_{1}\,\vert\,\dots\,\vert\, B_{k}$, where $\max B_{i}<\min B_{i+1}$.
Thus, any increasing chain on $[\hat{0},x_{0}]$ passes through this
single atom, and we have reduced the problem to Case 2.

\item Intervals of the form $[x_{m},x_{0}]$, where $x_{m}\neq\hat{0}$.

By Lemma \ref{lem:DDivProductStructure} and the discussion following,
it suffices to examine a single block $B$ of $x_{0}$. (The labels
on disjoint blocks are independent of each other.)

In $x_{m}$, let $B$ be subpartitioned as $B_{1}\,\vert\,\dots\,\vert\, B_{k}$,
with $\max B_{s}=b_{s}$ and $b_{1}<b_{2}<\dots<b_{k}$. The edges
we consider correspond with subpartitioning $B$ between itself and
$B_{1}\,\vert\,\dots\,\vert\, B_{k}$. 

First, we show that the lexicographically first chain $\mathbf{c}=x_{0}\dotsucc x_{1}\dotsucc\dots\dotsucc x_{m}$
is unique. If there are any negative labels down from $x_{i}$, the
edge $x_{i}\dotsucc x_{i+1}$ will have the label $-b_{s}$ with greatest
absolute value among negative labels. Thus, \[
x_{i+1}=x_{i}\wedge\left(B_{1}\dots B_{s}\,\vert\, B_{s+1}\dots B_{k}\right),\]
and hence $x_{i}\dotsucc x_{i+1}$ is the unique edge down from $x_{i}$
with this label. Otherwise, $x_{i}\dotsucc x_{i+1}$ will have the
least possible (positive) label, which is unique since $\vert\lambda\vert$
is a dual supersolvable $EL$-labeling on $[x_{m},x_{0}]$.

Next, we show that the lexicographically first chain is increasing.
Suppose that $\mathbf{c}$ has a descent at $x_{i-1}\dotsucc x_{i}\dotsucc x_{i+1}$,
with $\lambda(x_{i-1}\dotsucc x_{i})=\alpha$ and $\lambda(x_{i}\dotsucc x_{i+1})=\beta$,
corresponding to dividing a block $C$ as \[
C\dotsucc C_{1}\,\vert C_{2}\cup C_{3}\dotsucc C_{1}\,\vert\, C_{2}\,\vert\, C_{3}.\]
Since $\vert\lambda\vert$ is a dual $EL$-labeling, both labels cannot
be positive. Thus, $\beta<0$. If then $\vert\alpha\vert<\vert\beta\vert$,
we have $\max C_{1}<\max C_{2}<\min C_{3}$, and then $C\dotsucc C_{1}\cup C_{2}\,\vert\, C_{3}$
is noncrossing, with a $\beta$ label, and so lexicographically before
$x_{i-1}\dotsucc x_{i}$. Otherwise, $\vert\alpha\vert>\vert\beta\vert$.
Since we have a descent at $i$, we see $\alpha>0$, and so the $\pm\beta<\alpha$
label on the edge obtained by partitioning $C\dotsucc C_{1}\cup C_{2}\,\vert\, C_{3}$
is again lexicographically before $x_{i-1}\dotsucc x_{i}$. In either
case, we have shown that any $\mathbf{c}$ with a descent is not lexicographically
first.

Finally, we show that any increasing chain is lexicographically first.
Suppose that there is an edge $x_{0}\dotsucc y\,(\succ x_{m})$ that
receives a $-b_{s}$ label. Then $y=B_{1}\dots B_{s}\,\vert\, B_{s+1}\dots B_{k}$
is a non-crossing partition of $B$, and in particular $B_{s}<B_{s+1},\dots,B_{k}$.
We see that any subpartion of $x_{0}$ separating $B_{s}$ from $B_{t}$
for $t>s$ is non-crossing, thus every chain on $[x_{m},x_{0}]$ has
a $-b_{s}$ label. This fact, combined with (\ref{eq:SSPermutationProperty})
shows that any increasing chain on an interval must be constructed
inductively by repeatedly taking the least-labeled edge down, hence
be lexicographically first.\qedhere

\end{caseenv}
\end{proof}
The descending chains of Wachs's $EL$-labeling are $\{r_{\sigma}\,:\,\sigma\in A_{n}^{d}\}$,
where $r_{\sigma}$ corresponds to successively splitting $\sigma$
at the greatest possible $\sigma(id)$ \cite[Theorem 5.2]{Wachs:1996}.
It is easy to see that each $r_{\sigma}$ is also descending with
respect to our dual $EL$-labeling, and a dimension argument shows
us that $\{r_{\sigma}\,:\,\sigma\in A_{n}^{d}\}$ is exactly the set
of descending chains.

\subsection{An $EL$-ced for $\Pi_{n}^{d}$}

Order $\{\Sigma_{\alpha}\}$ lexicographically by the reverse of the
words $\alpha$ according to the reverse ordering on $[n]$. That
is, order lexicographically by the words $\alpha(n)\alpha(n-1)\cdots\alpha(1)$,
where $n\vartriangleleft n-1\vartriangleleft\cdots\vartriangleleft1$.
We refer to this ordering as \emph{$rr$-lex}, for {}``reverse reverse
lexicographic.'' For example, $132546$ is the first permutation
in $A_{6}^{2}$ with respect to $rr$-lex, while $231546\ltrr142536$
(since $4>3$ in position 5). 

We will prove the following version of Theorem \ref{thm:d-divisibleced_intro}.
Let $\Sigma_{\alpha}$ be as in the text preceding Theorem \ref{thm:d-divisible-homology-basis},
and $\lambda$ as in Section \ref{sub:d-divisible-EL-labeling}.
\begin{thm}
\label{thm:d-divisible-EL-ced} $\{\Sigma_{\alpha}\,:\,\alpha\in A_{n}^{d}\}$
ordered by $rr$-lex is a dual $EL$-ced of $\Pi_{n}^{d}$ with respect
to $\lambda$.
\end{thm}
We introduce some terms. If $B_{1}\,\vert\,\dots\,\vert\, B_{k}$
is a partition of $[n]$, then we say that $\alpha\in S_{n}$ has
the \emph{form} $B_{1}B_{2}\dots B_{k}$ if the first $\vert B_{1}\vert$
elements in the word $\alpha$ are in $B_{1}$, the next $\vert B_{2}\vert$
are in $B_{2}$, and so forth. When $k=2$, we say that $\alpha$
has \emph{switched form} $B_{1}B_{2}$ if $\alpha'$ has the form
$B_{1}B_{2}$ for $\alpha'=\alpha\circ(\vert B_{1}\vert\,\vert B_{1}\vert+1)$,
that is, for $\alpha'$ equal to $\alpha$ composed with the transposition
of adjacent elements at $\vert B_{1}\vert$. 

We can also talk of $\alpha$ having form $B_{1}B_{2}\dots B_{k}$
\emph{up to switching}, by which we mean some $\alpha'$ has the form
$B_{1}\dots B_{k}$, where $\alpha'$ is $\alpha$ up to transpositions
at the borders of some (but not necessarily all) of the blocks. Finally,
if $B\subseteq[n]$, then $\alpha\vert_{B}$ is the word $\alpha=\alpha(1)\alpha(2)\dots\alpha(n)$
with all $\alpha(i)$'s that are not in $B$ removed. 
\begin{example}
If $B_{1}=\{1,2,3\}$ and $B_{2}=\{4,5,6\}$, then $123456$, $321654$,
and $213465$ all have the form $B_{1}B_{2}$. $124356$ and $135246$
have switched form $B_{1}B_{2}$, while $152346$ does not have the
form $B_{1}B_{2}$, even up to switching.
\end{example}
Clearly, the $d$-divisible partition $B_{1}\,\vert\,\dots\,\vert\, B_{k}$
is in $\Sigma_{\alpha}$ if and only if $\alpha$ has the form $B_{1}B_{2}\dots B_{k}$
up to switching.
\begin{lem}
\label{lem:CEDcovers}Every maximal chain $\mathbf{c}$ in $\Pi_{n}^{d}$
is in $\Sigma_{\alpha}$ for some $\alpha\in A_{n}^{d}$. \end{lem}
\begin{proof}
We will in fact construct the earliest such $\alpha$ according to
the rr-lex ordering, which will in turn help us with Corollary \ref{cor:FormOfDeltaAlpha}.
The proof has a similar feel to the well-known quicksort algorithm.
Let $\mathbf{c}=\{\hat{1}=c_{0}\dotsucc\dots\dotsucc c_{n/d}=\hat{0}\}$.

Consider first the edge $\hat{1}\dotsucc c_{1}$ in $\mathbf{c}$.
The edge splits $[n]$ into $B_{1}\,\vert\, B_{2}$, and clearly such
$\alpha$, if it exists, must have the form $B_{1}B_{2}$ or $B_{2}B_{1}$
up to switching. If $\max B_{1}<\max B_{2}$, then all permutations
in $A_{n}^{d}$ of the (possibly switched) form $B_{1}B_{2}$ come
before permutations of the (possibly switched) form $B_{2}B_{1}$,
so the rr-lex first $\alpha$ with $\mathbf{c}$ in $\Sigma_{\alpha}$
has the form $B_{1}B_{2}$ up to switching.

Apply this argument inductively down the chain. At $c_{i}$, we will
have shown that the rr-lex first $\alpha$ with $\mathbf{c}$ in $\Sigma_{\alpha}$
must have the form $B_{1}B_{2}\cdots B_{i+1}$ up to switching. Then
if $c_{i}\dotsucc c_{i+1}$ splits block $B_{j}$ into $B_{j,1}$
and $B_{j,2}$, with $\max B_{j,1}<\max B_{j,2}$, an argument similar
to that with $\hat{1}\dotsucc c_{1}$ gives that $\alpha$ must in
fact have the form \[
B_{1}B_{2}\dots B_{j-1}B_{j,1}B_{j,2}B_{j+1}\dots B_{i+1}\]
 up to switching.

At the end, we have shown the earliest $\alpha$ having $\mathbf{c}$
in $\Sigma_{\alpha}$ must have the form $B_{1}\dots B_{n/d}$ up
to switching. Conversely, it is clear from the above that for any
$\alpha$ of this form, $\mathbf{c}$ is in $\Sigma_{\alpha}$. Sort
the elements of each $B_{i}$ in ascending order to get a permutation
$\alpha_{0}$. This $\alpha_{0}$ is in $S_{n}$ but not necessarily
in $A_{n}^{d}$, so we perform a switch at each $d$-divisible position
where there is an ascent (i.e., where $B_{i}<B_{i+1}$). This gives
us an element $\alpha\in A_{n}^{d}$ of the given form up to switching,
and finishes the proof of the statement.

\medskip{}
We continue nonetheless to finish showing that $\alpha$ is the first
element in $A_{n}^{d}$ with $\mathbf{c}$ in $\Sigma_{\alpha}$.
We need to show that if $\beta$ is another element of $A_{n}^{d}$
with the same form up to switching of $B_{1}B_{2}\dots B_{n/d}$ (but
different switches), then $\beta\gtrr\alpha$. If $B_{i}<B_{i+1}$,
then both $\alpha$ and $\beta$ are switched at $id$ (as otherwise
we are not in $A_{n}^{d}$). Otherwise, if $\beta$ is a switch at
$id$, then the switch exchanges $\beta(id)$ and $\beta(id+1)$ (up
to resorting the blocks). Since $\beta(id)>\beta(id+1)$, {}``unswitching''
moves a larger element of $[n]$ later in the permutation, yielding
an rr-lex earlier element of the given form up to switching.
\end{proof}
Let $\Delta_{\alpha}$ be the simplicial complex generated by maximal
chains that are in $\Sigma_{\alpha}$ ($\alpha\in A_{n}^{d}$), but
in no $\Sigma_{\beta}$ for $\beta\in A_{n}^{d}$ with $\beta\ltrr\alpha$.
In the following corollary, we summarize the information from the
proof of Lemma \ref{lem:CEDcovers} about the form of $\alpha$ with
$\mathbf{c}$ in $\Delta_{\alpha}$.
\begin{cor}
\label{cor:FormOfDeltaAlpha}Let $\mathbf{c}$ be a maximal chain
in $\Delta_{\alpha}$, with $y\dotsucc x$ an edge in $\mathbf{c}$
which merges blocks $B_{1}$ and $B_{2}$ into block $B$ ($\max B_{1}<\max B_{2}$).
Then
\begin{enumerate}
\item $\alpha$ has the form $\dots B_{1}B_{2}\dots$, up to switching.
\item Let $\tau_{id}$ be the transposition exchanging $id$ and $id+1$.
If $y\dotsucc x$ corresponds to a switch-and-split at $id$, then
the permutation $\alpha\circ\tau_{id}$ is ascending between positions
$(i-1)d+1$ and $(i+1)d$.
\item $\alpha\vert_{B_{1}}=\dots\max B_{1}$, i.e., $\max B_{1}$ is rightmost
in $\alpha\vert_{B_{1}}$.
\end{enumerate}
\end{cor}
\begin{proof}
(1) and (2) are clear from the proof of Lemma \ref{lem:CEDcovers}.

For (3), suppose that $\max B_{1}$ is not rightmost in $\alpha\vert_{B_{1}}$.
Then since $\alpha$ is ascending on $d$-segments, we have that $\max B_{1}$
is rightmost in some $d$-segment of $\alpha\vert_{B_{1}}$. If a
switch-and-split occurs at $\max B_{1}$ then we have a contradiction
of (2), while a split contradicts (1).
\end{proof}
Every chain passing through an atom $a$ has the same labels up to
sign, and Corollary \ref{cor:FormOfDeltaAlpha} tells us what the
labels are. It is now not difficult to prove (CLced-desc) and (CLced-bdry).
\begin{prop}
\label{pro:d-div-satisfiesCLced-desc}Let $[x_{m},x_{0}]$ be an interval
with $x_{m},x_{0}\in\Delta_{\alpha}$. Then there is at most one (dual)
descending maximal chain $\mathbf{c}$ on $[x_{m},x_{0}]$ which is
in $\Delta_{\alpha}$.\end{prop}
\begin{proof}
There are two cases:
\begin{caseenv}
\item $x_{m}=\hat{0}$.

It suffices to consider a block $B$ of $x_{0}$. Partitions of $B$
corresponding to edges in $\Sigma_{\alpha}$ must either split or
switch-and-split $\alpha\vert_{B}$ at $d$-divisible positions, and
as every chain on $[\hat{0},x_{0}]$ has bottommost label 0, all other
edges of a descending chain must have positive labels (and so correspond
to crossing partitions).
\begin{claim}
All edges of such a descending chain correspond to splittings of $\alpha$.\end{claim}
\begin{proof}
(of Claim) Suppose otherwise. Without loss of generality we can assume
that $x_{0}\dotsucc x_{1}$ in $\mathbf{c}$ corresponds to a switch-and-split
of $B$ into $B_{1}\,\vert\, B_{2}$, with $B$ the block of smallest
size which is switch-and-split by an edge in $\mathbf{c}$. We will
show that $\mathbf{c}$ has an ascent. 

Corollary \ref{cor:FormOfDeltaAlpha} part 2 tells us that the first
$d$ letters in $\alpha\vert_{B_{2}}$ are strictly greater than the
last $d$ in $\alpha\vert_{B_{1}}$, and since $\max B_{1}$ is rightmost
in $\alpha\vert_{B_{1}}$, that the first $d$ letters of $\alpha\vert_{B_{2}}$
are strictly greater than all of $B_{1}$. Since $\lambda(x_{m-1}\dotsucc\hat{0})=0$,
any negative label gives an ascent. If $\vert B_{2}\vert=d$, then
we have shown that $B_{1}\,\vert\, B_{2}$ is non-crossing (giving
a negative label). If on the other hand $\vert B_{2}\vert>d$, then
any subdivision of $B_{2}$ gives a label with absolute value $>\max B_{1}$,
hence an ascent. In either case, we contradict $\mathbf{c}$ being
a descending chain.
\end{proof}
It follows immediately that a descending chain on $[\hat{0},x_{0}]$
is unique.

\item $x_{m}\neq\hat{0}$.

As usual, we consider what happens to a block $B$ of $x_{0}$. In
$x_{m}$, let $B$ partition as $B_{1}\,\vert\,\dots\,\vert\, B_{k}$,
where $\alpha\vert_{B}$ has the form $B_{1}B_{2}\dots B_{k}$ up
to switching. Then every edge in $\Delta_{\sigma}$ comes from subdividing
at some $B_{i}$, i.e., as shown at the dotted line here\[
\dots\cup B_{j}\,\vert\, B_{j+1}\cup\dots\cup B_{i}\,\vdots\, B_{i+1}\cup\dots\cup B_{l}\,\vert\,\dots\]
Let $b_{i}=\max B_{i}$, so every chain on $[x_{m},x_{0}]$ has labels
$\pm b_{1},\pm b_{2},\dots,\pm b_{k-1}$.

Suppose that $\dots\cup B_{i}\,\vert\, B_{i+1}\cup\dots$ is crossing,
but $B_{i}\,\vert\, B_{i+1}$ is non-crossing. Corollary \ref{cor:FormOfDeltaAlpha}
part 3 tells us that $\max(\dots\cup B_{i})=\max B_{i}$, so that
\[
\min(B_{i+2}\cup\dots\cup B_{k})<\max B_{i}<\min B_{i+1}<\max B_{i+1}.\]
It follows that $B_{i+1}\,\vert\, B_{i+2}\cup\dots$ is also crossing.
Thus, if $B_{i}\,\vert\, B_{i+1}$ is non-crossing, then (positive)
$b_{i}$ is not the label of the first edge of a descending chain
$\mathbf{c}$, since $b_{i+1}>b_{i}$ would then be the label of a
later edge. That is, if $B_{i}\,\vert\, B_{i+1}$ is non-crossing,
then a descending chain has a $-b_{i}$ label. The {}``only if''
direction is immediate, thus there is a unique permutation and set
of signs for the $\pm b_{1},\dots,\pm b_{k-1}$ that could label a
descending chain.\qedhere

\end{caseenv}
\end{proof}
\begin{prop}
\label{pro:d-div-satisfiesCLced-bdry}Let $\mathbf{c}$ be a (non-maximal)
chain with extensions in both $\Sigma_{\alpha}$ and $\Sigma_{\beta}$,
$\beta\ltrr\alpha$. Then $\mathbf{c}$ has maximal extensions in
$\mathcal{M}(\Sigma_{\alpha})\setminus\mathcal{M}(\Delta_{\alpha})$.\end{prop}
\begin{proof}
Let $\mathbf{c}=\{\hat{1}=c_{0}>c_{1}\dots>c_{m}>c_{m+1}=\hat{0}\}$.
The first $\beta$ with $\mathbf{c}$ in $\Sigma_{\beta}$ obeys the
following two conditions:
\begin{enumerate}
\item For each $c_{i}\neq\hat{0}$, each block $B$ in $c_{i-1}$ splits
into sub-blocks $B_{1},\dots,B_{k}$ in $c_{i}$, where $\max B_{1}<\dots<\max B_{k}$.
The restriction $\beta\vert_{B}$ is of the form $B_{1}B_{2}\dots B_{k}$.
(By repeated application of Corollary \ref{cor:FormOfDeltaAlpha}
.)
\item For each block $B$ of $c_{m}$, $\beta\vert_{B}$ is the permutation
\[
\{b_{1}b_{2}\dots b_{d+1}b_{d}\dots b_{id+1}b_{id}\dots b_{k}\},\]
where $B=\{b_{1},\dots,b_{k}\}$ for $b_{1}<\dots<b_{k}$. That is,
$\beta\vert_{B}$ is the ascending permutation of the elements of
$B$, with transpositions applied at $d$-divisible positions. (The
proof is by starting at the end and working to the front, greedily
taking the greatest possible element for each position.) 
\end{enumerate}
Since $\alpha$ is not the first permutation in $A_{n}^{d}$ such
that $\mathbf{c}\in\Sigma_{\alpha}$, $\alpha$ must violate at least
one of these. If it violates (1) for some $B$, then $\alpha\vert_{B}$
has the form $B_{1}\dots B_{k}$ with $\max B_{j}>\max B_{j+1}$.
Merge $B_{j}$ and $B_{j+1}$ to add an edge down from $c_{i-1}$
that is in $\Sigma_{\alpha}$, otherwise extend arbitrarily in $\Sigma_{\alpha}$.
By Corollary \ref{cor:FormOfDeltaAlpha} part 1, the resulting chain
is not in $\Delta_{\alpha}$. 

If $\alpha\vert_{B}$ violates (2) for some $B$, then extend $\mathbf{c}$
by switch-and-splitting at every $d$-divisible position of $B$,
otherwise arbitrarily in $\Sigma_{\alpha}$. At the bottom, $B$ is
partitioned into some $B_{1}\,\vert\,\dots\,\vert\, B_{k}$. Since
(2) is violated, applying transpositions to $\alpha$ at $d$-divisible
partitions gives a descent. But this contradicts the conclusion of
Corollary \ref{cor:FormOfDeltaAlpha} part 2, and the resulting chain
is not in $\Delta_{\alpha}$. 
\end{proof}
We check the $CL$-ced properties: Wachs had already proved (CLced-polytope)
as presented in Theorem \ref{thm:d-divisible-homology-basis}, Lemma
\ref{lem:CEDcovers} gives us (CLced-union), Proposition \ref{pro:d-div-satisfiesCLced-desc}
gives (CLced-desc), and Proposition \ref{pro:d-div-satisfiesCLced-bdry}
gives (CLced-bdry). We have completed the proof of Theorem \ref{thm:d-divisible-EL-ced}.

\section{\label{sec:coset-lattice}The coset lattice}

\subsection{Group theory background}

The \emph{coset poset} of $G$, denoted $\cosetposet(G)$, is the
set of all right cosets of all proper subgroups of $G$, ordered under
inclusion. The \emph{coset lattice} of $G$, denoted $\cosetlat(G)$,
is $\cosetposet(G)\cup\{\emptyset,G\}$, that is, $\cosetposet(G)$
with a top $\hat{1}=G$ and bottom $\hat{0}=\emptyset$ added. With
our definitions, it makes sense to look at the order complex of $\cosetlat(G)$
(which is the set of all chains of $\cosetposet(G)$), and so we talk
about the coset lattice, even though {}``coset poset'' has a better
sound to it. We notice that $\cosetlat(G)$ has meet operation $Hx\wedge Ky=Hx\cap Ky$
and join $Hx\vee Ky=\langle H,K,xy^{-1}\rangle y$ (so it really is
a lattice.) General background on the coset lattice can be found in
\cite[Chapter 8.4]{Schmidt:1994}, and its topological combinatorics
have been studied in \cite{Brown:2000,Ramras:2005,Woodroofe:2007}.

The \emph{subgroup lattice} of $G$, denoted $L(G)$ is the set of
all subgroups of $G$, ordered by inclusion. General background can
be found in \cite{Schmidt:1994}, and its topological combinatorics
have been studied extensively, for example in \cite{Shareshian:2001,Thevenaz:1985}.

Notice that for any $x\in G$, the interval $[x,G]$ in $\cosetlat(G)$
is isomorphic to $L(G)$. It is a theorem of Iwasawa \cite{Iwasawa:1941}
that $L(G)$ is graded if and only if $G$ is supersolvable, hence
$\cosetlat(G)$ is graded under the same conditions. As we have only
defined convex ear decompositions for pure complexes, we are primarily
interested in supersolvable groups in this paper.

Schweig proved the following:
\begin{prop}
\emph{(Schweig \cite{Schweig:2006})} For a supersolvable lattice
$L$, the following are equivalent:
\begin{enumerate}
\item $L$ has a convex ear decomposition.
\item $L$ is doubly Cohen-Macaulay.
\item Every interval of $L$ is complemented.
\end{enumerate}
\end{prop}
\begin{note}
A construction very much like Schweig's convex ear decomposition was
earlier used by Thévenaz in \cite{Thevenaz:1985} on a subposet of
$L(G)$ to understand the homotopy type and the conjugation action
on homology of $L(G)$ for a solvable group $G$. 
\end{note}
It is easy to check that any normal subgroup $N\normalin G$ is left
modular in $L(G)$, so a supersolvable group has a supersolvable subgroup
lattice with any chief series as its left modular chain. Let $G'$
denote the commutator subgroup of $G$. The following collected classification
of groups with every interval in their subgroup lattice complemented
is presented in Schmidt's book \cite[Chapter 3.3]{Schmidt:1994},
and was worked out over several years by Zacher, Menegazzo, and Emaldi.
\begin{prop}
\label{pro:RKgroupProperties}The following are equivalent for a (finite)
group $G$:
\begin{enumerate}
\item Every interval of $L(G)$ is complemented.
\item If $H$ is any subgroup on the interval $[H_{0},H_{1}]$, then there
is a $K$ such that $HK=H_{1}$ and $H\cap K=H_{0}$.
\item $L(G)$ is coatomic, i.e., every subgroup $H$ of $G$ is an intersection
of maximal subgroups of $G$.
\item $G$ has elementary abelian Sylow subgroups, and if $H_{1}\normalin H_{2}\normalin H_{3}\subseteq G$,
then $H_{1}\normalin H_{3}$.
\item $G'$ and $G/G'$ are both elementary abelian, $G'$ is a Hall $\pi$-subgroup
of $G$, and every subgroup of $G'$ is normal in $G$.
\end{enumerate}
\end{prop}
\begin{note}
The classification of finite simple groups is used in the proof that
(3) is equivalent to the others.
\end{note}
We will follow Schmidt and call such a group a \emph{relatively complemented}
group. 

We notice that relatively complemented groups are \emph{complemented},
that is, satisfy the condition of Proposition \ref{pro:RKgroupProperties}
Part (2) on the interval $[1,G]$. On the other hand, $S_{3}\times\mathbb{Z}_{3}$
is an example of a complemented group which is not relatively complemented.
The complemented groups are exactly the groups with (equivalently
in this case) shellable, Cohen-Macaulay, and sequentially Cohen-Macaulay
coset lattice \cite{Woodroofe:2007}. Computation with GAP \cite{GAP4}
shows that there are 92804 groups of order up to 511, but only 1366
complemented groups, and 1186 relatively complemented groups.

We summarize the situation for the subgroup lattice regarding convex
ear decompositions:
\begin{cor}
The following are equivalent for a group $G$:
\begin{enumerate}
\item $L(G)$ has a convex ear decomposition.
\item $L(G)$ is doubly Cohen-Macaulay.
\item G is a relatively complemented group.
\end{enumerate}
\end{cor}
As a consequence, we get one direction of Theorem \ref{thm:cosetposetced_intro}.
\begin{cor}
\label{cor:2CMimpliesRC}If $\cosetlat(G)$ is doubly Cohen-Macaulay
(hence if it has a convex ear decomposition), then $G$ is a relatively
complemented group.\end{cor}
\begin{proof}
Every interval of a 2-Cohen-Macaulay poset is 2-Cohen-Macaulay, and
the interval $[1,G]$ in $\cosetlat(G)$ is isomorphic to $L(G)$.
\end{proof}
The remainder of Section \ref{sec:coset-lattice} will be devoted
to proving the other direction.

\subsection{A dual $EL$-labeling for $\cosetlat(G)$\label{sub:EL-labelingForC(G)}}

As with the $d$-divisible partition lattice, the first thing we need
is a dual $EL$-labeling of $\cosetlat(G)$. We will construct one
for the more general case where $G$ is complemented. The main idea
is to start with the $EL$-labeling of an upper interval and {}``twist''
by adding signs, similarly to our $EL$-labeling for $\Pi_{n}^{d}$.
The resulting labeling is significantly simpler than the one I described
in \cite{Woodroofe:2007}.

Let $G$ be a complemented group, and fix a chief series $G=N_{1}\innormal N_{2}\innormal\cdots\innormal N_{k+1}=1$
for $G$ throughout the remainder of Section \ref{sec:coset-lattice}.
Our labeling (and later our convex ear decomposition) will depend
on this choice of chief series, but the consequences for the topology
and $h$-vector of $\cosetlat(G)$ will obviously depend only on $G$.

For each factor $N_{i}/N_{i+1}$, choose a complement $B_{i}^{0}$,
i.e., a subgroup such that $N_{i}B_{i}^{0}=G$ but $N_{i}\cap B_{i}^{0}=N_{i+1}$.
(Such a $B_{i}^{0}$ exists, as every quotient group of a complemented
group is itself complemented \cite[Lemma 3.2.1]{Schmidt:1994}.) From
Section \ref{sub:Supersolvable-lattices}, the usual dual $EL$-labeling
of the subgroup lattice of a supersolvable group is \[
\lambda^{ss}(K_{0}\dotsupset K_{1})=\max\{i\,:\, N_{i}K_{1}\supseteq K_{0}\}=\min\{i\,:\, N_{i+1}\cap K_{0}\subseteq K_{1}\}.\]
Remember that $\lambda^{ss}$ labels every chain on a given interval
with the same set of labels (up to permutation). 

We now define a labeling $\lambda$ of $\cosetlat(G)$ as follows.
For $K_{0}\dotsupset K_{1}$ labeled by $\lambda^{ss}$ with $i$,
let \begin{align*}
\lambda(K_{0}x\dotsupset K_{1}x) & =\left\{ \begin{array}{cl}
-i & \textrm{ if }K_{1}x=K_{0}x\cap B_{i}^{0},\\
\,\,\,\, i & \textrm{otherwise, and}\end{array}\right.\\
\lambda(x\dotsupset\emptyset) & =0.\end{align*}
It is immediate from this construction that $|\lambda|_{[x,G]}=\lambda^{ss}$
(up to the {}``dropping $x$'' isomorphism), much like the situation
discussed in Section \ref{sub:d-divisible-EL-labeling} for the $d$-divisible
partition lattice. 
\begin{lem}
\label{lem:ComplementsVsFactors}Let $G$ be any supersolvable group.
Then:
\begin{enumerate}
\item If $KB=G$ where $B\subsetdot G$, then $K\cap B\subsetdot K$.
\item If $\lambda^{ss}(K_{0}\dotsupset K_{1})=i$, then for any complement
$B_{i}$ of $N_{i}/N_{i+1}$ and $K\supseteq K_{0}$ we have $K_{0}B_{i}=KB_{i}=G$.
\end{enumerate}
\end{lem}
\begin{proof}
For part 1, count: $\vert K\cap B\vert=\frac{\vert K\vert\vert B\vert}{\vert G\vert}=\frac{\vert K\vert}{[G:B]}$
and by supersolvability, $[K:K\cap B]=[G:B]$ is a prime.

For part 2, by the definition of the labeling, $N_{i}\cap K_{0}\not\subseteq K_{1}$
but $N_{i+1}\cap K_{0}\subseteq K_{1}$. We see that $N_{i}\cap K_{0}\not\subseteq N_{i+1}$,
and since $N_{i+1}\subsetdot N_{i}$, that $(N_{i}\cap K_{0})N_{i+1}=N_{i}$
and so $K_{0}N_{i+1}\supseteq N_{i}$. Then $K_{0}B_{i}=K_{0}N_{i+1}B_{i}\supseteq N_{i}B_{i}=G$.\end{proof}
\begin{thm}
\label{thm:anELlabelingofC(G)}If $G$ is a complemented group, then
$\lambda$ is a dual $EL$-labeling of $\cosetlat(G)$.\end{thm}
\begin{proof}
We need to show that every interval has a unique increasing maximal
chain which is lexicographically first. There are two kinds of intervals
we need to check:
\begin{caseenv}
\item $[\emptyset,H_{0}x]$

As the last label of any chain on this interval is 0, in an increasing
chain the others must be negative (in increasing order). Since every
chain has the same labels up to permutation, uniqueness of the increasing
chain is clear from the definition of $\lambda$. Existence follows
from applying Lemma \ref{lem:ComplementsVsFactors} to the maximal
subgroups $B_{i}^{0}$. Finally, the chain takes the edge with the
least possible label down from each $Hx$, so it is lexicographically
first.

\item $[H_{n}x,H_{0}x]$

Let $S$ be the label set of $\lambda^{ss}$ restricted to the interval
$[H_{n},H_{0}]$. We notice that a $-i$ label is possible on $[H_{n}x,H_{0}x]$
only if $H_{n}x\subseteq B_{i}^{0}$ and $i\in S$. Thus, the lexicographically
first chain is labeled by all possible negative labels (in increasing
order), followed by the remaining (positive) labels, also in increasing
order. Such a chain clearly exists and is increasing. The negative-labeled
part is unique since a $-i$ label corresponds with intersection by
$B_{i}^{0}$, while the positive-labeled part is unique since $\lambda^{ss}=\vert\lambda\vert$
is an $EL$-labeling.

It remains to check that there are no other increasing chains. We
have already shown that there is only one increasing chain which has
a $-i$ label for each $B_{i}^{0}$ containing $H_{n}x$, so any other
increasing chain would need to have a $+i$ label for some $i\in S$
where $H_{n}x\subseteq B_{i}^{0}$. Without loss of generality, let
this edge $H_{0}x\dotsupset H_{1}x$ be directly down from $H_{0}x$.
Then $i=\min S$, and since $\lambda^{ss}$ is a dual $EL$-labeling,
we have that there is a unique edge down from $H_{0}x$ with label
$\pm i$. But then $H_{1}x=H_{0}x\cap B_{i}^{0}$, so the edge gets
a $-i$ label, giving us a contradiction and completing the proof.\qedhere

\end{caseenv}
\end{proof}
Though we do not need it for our convex ear decomposition, let us
briefly sketch the decreasing chains of $\lambda$. Following Thévenaz
\cite{Thevenaz:1985}, a \emph{chain of complements} to a chief series
$G=N_{1}\dotsupset N_{2}\dotsupset\dots\dotsupset N_{k+1}=1$ is a
chain of subgroups $G=H_{k+1}\dotsupset H_{k}\dotsupset\dots\dotsupset H_{1}=1$
where for each $i,$ $H_{i}$ is a complement to $N_{i}$. Thévenaz
showed that the chains of complements in $G$ correspond to homotopy
spheres in $\vert L(G)\vert$. The following proposition is the $EL$-shelling
version of Thévenaz's result for a supersolvable group, and is a special
case of \cite[Proposition 4.3]{Woodroofe:2008}.
\begin{prop}
The decreasing chains in $L(G)$ with respect to $\lambda^{ss}$ are
the chains of complements to $G=N_{1}\innormal\dots\innormal N_{k+1}=1$.\end{prop}
\begin{proof}
If $G=H_{k+1}\dotsupset H_{k}\dotsupset\dots\dotsupset H_{1}=1$ is
a chain of complements, then $N_{i}H_{i}=G\supseteq H_{i+1}$, while
\[
N_{i+1}H_{i}\cap H_{i+1}=(N_{i+1}\cap H_{i+1})H_{i}=1\cdot H_{i}=H_{i}\]
by left modularity (the Dedekind identity). Thus $\lambda^{ss}(H_{i+1}\dotsupset H_{i})=i$,
and the chain is descending.

Conversely, any descending chain corresponds to a sphere in $\vert L(G)\vert$,
and by Thévenaz's correspondence, there can be no others.\end{proof}
\begin{cor}
The decreasing chains in $\cosetlat(G)$ with respect to $\lambda$
are all cosets of chains of complements $\{G=H_{k+1}x\dotsupset\dots\dotsupset H_{1}x=x\dotsupset\emptyset\}$
to the chief series $G=N_{1}\innormal\dots\innormal N_{k+1}=1$ such
that no $H_{i}x=H_{i+1}x\cap B_{i}^{0}$.
\end{cor}

\subsection{A convex ear decomposition for $\cosetlat(G)$}

Recall that subgroups $H$ and $K$ \emph{commute} if $HK=KH$ is
a subgroup of $G$. 
\begin{lem}
\label{lem:ComplementsCommute}\emph{(Warm-up Lemma)} Let $G$ be
a solvable group with chief series $G=N_{1}\innormal N_{2}\innormal\cdots\innormal N_{k+1}=1$,
and $B_{i}$ and $B_{j}$ be complements of normal factors $N_{i}/N_{i+1}$
and $N_{j}/N_{j+1}$ where $i\neq j$. Then $B_{i}$ and $B_{j}$
commute.\end{lem}
\begin{proof}
Suppose $j<i$. Then $N_{i+1}\subsetneq N_{i}\subseteq N_{j+1}\subsetneq N_{j}$,
and $N_{j+1}\subseteq B_{j}$. Thus, $B_{j}B_{i}\supseteq N_{i}B_{i}=G$.
\end{proof}
Recalling $G=N_{1}\innormal N_{2}\innormal\cdots\innormal N_{k+1}=1$
as the chief series we fixed in Section \ref{sub:EL-labelingForC(G)},
let \[
\mathcal{B}=\{B_{i}\,:\, B_{i}\mbox{ is a complement to }N_{i}/N_{i+1},1\leq i\leq k\}\]
be a set of complements to $N_{i}$, one complement for each chief
factor (so that $\vert\mathcal{B}\vert=k$). For any $x\in G$, let
$\mathcal{B}x=\{B_{i}x\,:\, B_{i}\in\mathcal{B}\}$. We will call
$\mathcal{B}$ a \emph{base-set} for $\cosetlat(G)$. 

The first step is to show that intersections of certain cosets of
$\mathcal{B}$ give us a cube, using a stronger version of Lemma \ref{lem:ComplementsCommute}.
\begin{lem}
If $\mathcal{B}$ is a base-set, then $(B_{i_{1}}\cap\dots\cap B_{i_{l}})B_{i_{\ell+1}}=G$.\end{lem}
\begin{proof}
We count

\begin{eqnarray*}
\vert(B_{i_{1}}\cap\dots\cap B_{i_{\ell}})B_{i_{\ell+1}}\vert & = & \frac{\vert(B_{i_{1}}\cap\dots\cap B_{i_{\ell}})\vert\vert B_{i_{\ell+1}}\vert}{\vert B_{i_{1}}\cap\dots\cap B_{i_{\ell}}\cap B_{i_{\ell+1}}\vert}\\
 & = & \frac{\vert B_{i_{1}}\cap\dots\cap B_{i_{\ell-1}}\vert\vert B_{i_{\ell}}\vert\vert B_{i_{\ell+1}}\vert}{\vert(B_{i_{1}}\cap\dots\cap B_{i_{\ell-1}})B_{i_{\ell}}\vert\vert B_{i_{1}}\cap\dots\cap B_{i_{\ell}}\cap B_{i_{\ell+1}}\vert}.\end{eqnarray*}

By induction on $\ell$, this is \[
=\frac{\vert B_{i_{1}}\cap\dots\cap B_{i_{\ell-1}}\vert\vert B_{i_{\ell}}\vert\vert B_{i_{\ell+1}}\vert}{\vert G\vert\vert B_{i_{1}}\cap\dots\cap B_{i_{\ell}}\cap B_{i_{\ell+1}}\vert},\]

and by symmetry, \[
\vert(B_{i_{1}}\cap\dots\cap B_{i_{\ell}})B_{i_{\ell+1}}\vert=\vert(B_{i_{1}}\cap\dots\cap B_{i_{\ell-1}}\cap B_{i_{\ell+1}})B_{i_{\ell}}\vert.\]
Repeating this argument shows that $\vert(B_{i_{1}}\cap\dots\cap B_{i_{\ell}})B_{i_{\ell+1}}|$
is independent of the ordering of the $B_{i_{j}}$'s, or of the choice
of $i_{\ell+1}$.

Then take $i_{\ell+1}$ to be the largest index of any such $B_{i_{j}}$,
so that $N_{i_{\ell+1}}\subseteq B_{i_{1}}\cap\dots\cap B_{i_{\ell}}$.
In particular, $(B_{i_{1}}\cap\dots\cap B_{i_{\ell}})B_{i_{\ell+1}}\supseteq N_{i_{\ell+1}}B_{i_{\ell+1}}=G$.
Since the ordering of the $i_{j}$'s doesn't affect the cardinality,
$\vert(B_{i_{1}}\cap\dots\cap B_{i_{\ell}})B_{i_{\ell+1}}|=|G|$ for
any choice of $i_{\ell+1}$, proving the lemma.\end{proof}
\begin{cor}
\label{cor:BasisForC(G)GivesCube}If $\mathcal{B}$ is a base-set
and $x$ is such that the elements of $\mathcal{B}$ and $\mathcal{B}x$
are distinct from one another (i.e., $B_{i}\neq B_{i}x$ for all $i$),
then the meet sublattice generated by $\mathcal{B}\cup\mathcal{B}x$
is isomorphic to the face lattice of the boundary of a $k$-cube.\end{cor}
\begin{proof}
Any $B_{j}$ commutes with any intersection of $B_{i}$'s, $j\neq i$,
and the result follows from Lemma \ref{lem:ComplementsVsFactors}
and since $B_{i}\cap B_{i}x=\emptyset$ for all $i$.
\end{proof}
We henceforth assume that $G$ is relatively complemented. 

Let $\mathcal{B}$ be a base-set for $\cosetlat(G)$ as above, and
$x\in G$ be such that $B_{i}x\neq B_{i}^{0}$ (for each $i$). Then
we define $\Sigma_{\mathcal{B}x}$ to be the meet sublattice of $\cosetlat(G)$
generated by $\mathcal{B}x\cup\{B_{i}^{0}\,:\, B_{i}x=B_{i}^{0}x\}$,
and the larger meet sublattice $\Sigma_{\mathcal{B}x}^{+}$ to be
generated by \[
\mathcal{B}x\cup\{B_{i}^{0}\,:\, B_{i}x=B_{i}^{0}x\}\cup\{B_{i}y_{i}\,:\, B_{i}\neq B_{i}^{0}\},\]
where the $y_{i}$'s are some elements such that $B_{i}y_{i}\neq B_{i}x$.
By Lemma \ref{lem:ComplementsVsFactors} and the proof of Corollary
\ref{cor:BasisForC(G)GivesCube}, \[
\bigcap_{\{i\,:\, B_{i}x=B_{i}^{0}x\}}B_{i}^{0}\cap\bigcap_{\{i\,:\, B_{i}\neq B_{i}^{0}\}}B_{i}y_{i}=y\mbox{ (for some }y\mbox{),}\]
so $\Sigma_{\mathcal{B}x}^{+}$ is given by all intersections of $\mathcal{B}x\cup\mathcal{B}y$.
Thus (also by Corollary \ref{cor:BasisForC(G)GivesCube}) $\vert\Sigma_{\mathcal{B}x}^{+}\vert$
is a convex polytope with subcomplex $\vert\Sigma_{\mathcal{B}x}\vert$.
\begin{lem}
Let $H_{0}x\dotsupset H_{1}x$ be an edge in $\cosetlat(G)$ with
$\lambda(H_{0}x\dotsupset H_{1}x)=i$. Then $H_{1}x=H_{0}x\cap B_{i}x$
for some complement $B_{i}$ to $N_{i}/N_{i+1}$.\end{lem}
\begin{proof}
Since every maximal chain in $\cosetlat(G)$ has exactly one edge
with $\lambda(H_{0}x\dotsupset H_{1}x)=\pm i$ for each $i\in[k]$,
it suffices to show that $H_{1}$ is contained in some complement
$B_{i}$ to $N_{i}/N_{i+1}$. Then $H_{0}$ cannot be contained in
$B_{i}$, as that would give two $\pm i$ edges, and so $H_{1}=H_{0}\cap B_{i}$.

Since $G$ is relatively complemented, every interval in $L(G)$ is
complemented. In particular, any interval of height 2 has both increasing
and decreasing chains, so for any $H_{-1}\dotsupset H_{0}$ there
is an $H_{1}^{+}\subsetdot H_{-1}$ with $\lambda^{ss}(H_{-1}\dotsupset H_{1}^{+})=i$. 

Repeat this argument inductively on $H_{-1}\dotsupset H_{1}^{+}$
until $H_{-1}=G$. The final $H_{1}^{+}$ is the desired $B_{i}$,
and the definition of $\lambda^{ss}$ shows that $B_{i}$ is a complement
to $N_{i}/N_{i+1}$.\end{proof}
\begin{cor}
\label{cor:EarsCover_C(G)}Every maximal chain in $\cosetlat(G)$
is in some $\Sigma_{\mathcal{B}x}$. 
\end{cor}
Corollary \ref{cor:EarsCover_C(G)} would not hold if we replaced
`relatively complemented' with any weaker condition, since the result
implies coatomicity, and Proposition \ref{pro:RKgroupProperties}
tells us that relatively complemented groups are exactly those with
coatomic subgroup lattice.

\medskip{}

Now that we have a set of cubes that cover $\cosetlat(G)$, the next
step is to assign an order to them. For any base-set $\mathcal{B}$,
let $\rho_{i}(\mathcal{B})$ be 0 if $B_{i}=B_{i}^{0}$, and 1 otherwise.
We put the $\rho_{i}$'s together in a binary vector $\rho(\mathcal{B})$,
which we will call the \emph{pattern} of $\mathcal{B}$. Order the
$\mathcal{B}x$'s (and hence the $\Sigma_{\mathcal{B}x}$'s) in any
linear extension of the lexicographic order on $\rho(\mathcal{B})$.
Let $\Delta_{\mathcal{B}x}$ be the simplicial complex with facets
the maximal chains that are in $\Sigma_{\mathcal{B}x}$, but not in
any preceding $\Sigma_{\mathcal{B}'x'}$.

The $\Sigma_{\mathcal{B}x}$'s are generally proper subsets of face
lattices of convex polytopes, so (CLced-polytope) does not hold and
we do not have an $EL$-ced. We can use the same sort of argument,
however, to prove the following refinement of Theorem \ref{thm:cosetposetced_intro}:
\begin{thm}
\label{thm:CEDforCosetLattice}$\{\Delta_{\mathcal{B}x}\}$ is a convex
ear decomposition for $\cosetlat(G)$ under the pattern ordering.
\end{thm}
Corollary \ref{cor:EarsCover_C(G)} shows that the ears cover $\cosetlat(G)$,
that is, that (CLced-union) holds. Our next step is to show that an
analogue of (CLced-desc) holds.

It will be convenient to let $S([a,b])$ be the label set of $\vert\lambda\vert$
on the interval $[a,b]$, that is, the set of nonnegative $i$'s such
that $\lambda$ gives $\pm i$ labels on cover relations in $[a,b]$.
\begin{lem}
\label{lem:C(G)CedHasDescentProperty}For any interval $[a,b]$ in
$\cosetlat(G)$, there is at most one (dual) descending maximal chain
$\mathbf{c}$ on $[a,b]$ which is in $\Delta_{\mathcal{B}x}$. \end{lem}
\begin{proof}
If $a=\emptyset$, then the unique descending chain on $[a,b]$ in
$\Sigma_{\mathcal{B}x}$ is given by intersecting with each $B_{i}x$
($i\in S([\emptyset,b])\setminus\{0\}$) in order.

If $a\neq\emptyset$, then the interval $[a,b]$ in $\Sigma_{\mathcal{B}x}$
is Boolean, with a maximal chain for each permutation of $S([a,b])$.
If there is a $-i$ label on a chain in $\Delta_{\mathcal{B}x}$,
then the edge can be obtained by intersecting with $B_{i}^{0}$. But
since $\Sigma_{\mathcal{B}x}$ is the first such complex containing
the chain, we must have $\rho_{i}(\mathcal{B})=0$ (otherwise, replace
$B_{i}x$ with $B_{i}^{0}$). But this tells us that every label with
absolute value $i$ on $[a,b]$ in $\Delta_{\mathcal{B}x}$ is negative.
Every such chain thus has the same set of labels, and at most one
permutation of these labels is descending.\end{proof}
\begin{cor}
\label{cor:CosetEarsAreShellable}$\Delta_{\mathcal{B}x}$ is shellable.\end{cor}
\begin{proof}
Suppose a maximal chain $\mathbf{c}=\{G=c_{1}\dotsupset\dots\dotsupset c_{k+1}\dotsupset c_{k+2}=\emptyset\}$
in $\Sigma_{\mathcal{B}x}$ has an ascent at $j$. If $j\neq k+1$,
then it is immediate that $\mathbf{c}\setminus\{c_{j}\}$ has two
extensions in $\Sigma_{\mathcal{B}x}$, and we argue exactly as in
Lemma \ref{lem:CLcedReplaceAscentWithDescent} and Theorem \ref{thm:ELcedsAreCeds}.

If $j=k+1$, then the ascent at $j$ has labels $-i,0$, and hence
$\rho_{i}(\mathcal{B})=0$ and $\Sigma_{\mathcal{B}x}$ is the first
cube containing $\mathbf{c}\setminus\{c_{k+1}\}$. Intersecting with
$B_{i}^{0}x$ instead of $B_{i}^{0}$ at $c_{k}$ gives another chain
$\mathbf{c}'$ in $\Delta_{\mathcal{B}x}$ with a descent at $k+1$,
and we again argue as in Theorem \ref{thm:ELcedsAreCeds}. 
\end{proof}
Finally, we show directly that (ced-bdry) holds. We start with a lemma.
\begin{lem}
\label{lem:earchain_extension} Given any chain $\mathbf{c}=\{G=c_{1}\supset\cdots\supset c_{m}\supset c_{m+1}=\emptyset\}$,
there is an extension to a maximal chain $\mathbf{c}^{++}$ such that
if $\mathbf{c}$ is in $\Sigma_{\mathcal{B}x}$, then $\mathbf{c}^{++}$
is in some $\Sigma_{\mathcal{B}x}^{+}$. If $\mathcal{B}x$ is the
first such with $\mathbf{c}$ in $\Sigma_{\mathcal{B}x}$, then $\mathbf{c}^{++}$
is in $\Sigma_{\mathcal{B}x}$.\end{lem}
\begin{proof}
We make the extension in two steps. First, let $\mathbf{c}^{+}$ be
the extension of $\mathbf{c}$ by augmenting each $c_{j}\supset c_{j+1}$
for $j\neq m$ with the chain on $[c_{j+1},c_{j}]$ that is increasing
according to $\vert\lambda\vert$. Intersecting $c_{j}$ iteratively
with $B_{i}x$ or $B_{i}^{0}$ (as appropriate, for each $i$ in $S([c_{j+1},c_{j}])$)
in increasing order gives this chain, thus, $\mathbf{c}^{+}$ is also
in $\Sigma_{\mathcal{B}x}$.

In a similar manner, let $\mathbf{c}^{++}$ be the extension of $\mathbf{c}^{+}$
at $c_{m}\supset\emptyset$ by intersecting with each $B_{i}^{0}$
for $i\in S(m)$ in increasing order. Suppose $c_{m}=Hx$. Then uniqueness
of the lexicographically first chain in $[1,H]$ gives that $H\cap B_{i}^{0}=H\cap B_{i}$,
so there is some $B_{i}y_{i}$ with $Hx\cap B_{i}^{0}=Hx\cap B_{i}y_{i}$.
Repeated use of this gives us a $\Sigma_{\mathcal{B}x}^{+}$ containing
$\mathbf{c}^{++}$: the generating elements for this cube include
those for $\Sigma_{\mathcal{B}x}$ and the $B_{i}y_{i}$'s found here.
Notice that if $\rho(i)=0$ for each $i\in S([\emptyset,c_{m}])$,
then $B_{i}^{0}$ is already in the generating set for $\Sigma_{\mathcal{B}x}$,
thus $\mathbf{c}^{++}$ is also in $\Sigma_{\mathcal{B}x}$.\end{proof}
\begin{prop}
\label{pro:CedBdryForC(G)}$\Delta_{\mathcal{B}x}\cap\left(\bigcup_{\mathcal{B}'x'\prec\mathcal{B}x}\Delta_{\mathcal{B}'x'}\right)=\bdry\Delta_{\mathcal{B}x}$.\end{prop}
\begin{proof}
Suppose that $\mathbf{c}$ is in $\Delta_{\mathcal{B}x}\cap\left(\bigcup_{\mathcal{B}'x'\prec\mathcal{B}x}\Delta_{\mathcal{B}'x'}\right)$,
and let $\mathbf{c}^{++}$ be as in Lemma \ref{lem:earchain_extension}.
Then $\mathbf{c}^{++}$ is an extension in $\Sigma_{\mathcal{B}x}^{+}$,
but since $\mathbf{c}^{++}$ is contained in $\Sigma_{\mathcal{B}'x}$
for the first such complex containing $\mathbf{c}$, we get that $\mathbf{c}^{++}$
is in $\mathcal{M}(\Sigma_{\mathcal{B}x}^{+})\setminus\mathcal{M}(\Delta_{\mathcal{B}x})$.
Lemma \ref{lem:SubmanifoldBdry} then gives that $\mathbf{c}$ is
in $\bdry\Delta_{\mathcal{B}x}$.

Conversely, let $\mathbf{c}$ be in $\Delta_{\mathcal{B}x}$, but
not in a previous $\Sigma_{\mathcal{B}'x'}$. Since $\mathbf{c}$
is not in any previous $\Sigma_{\mathcal{B}'x'}$, no extensions of
it are either, so any extension of $\mathbf{c}$ that is in $\Sigma_{\mathcal{B}x}$
is in $\Delta_{\mathcal{B}x}$. As we have ordered the base-sets by
pattern, we get that $\rho_{i}(\mathcal{B})=0$ for $i\in S([\emptyset,c_{m}])$,
thus, by the special treatment of $B_{i}^{0}$ in the definition of
$\Sigma_{\mathcal{B}x}$, every extension of $\mathbf{c}$ in any
$\Sigma_{\mathcal{B}x}^{+}$ is in $\Sigma_{\mathcal{B}x}$. Combining
these two statements, we see that there is no extension of $\mathbf{c}$
in $\mathcal{M}(\Sigma_{\mathcal{B}x}^{+})\setminus\mathcal{M}(\Delta_{\mathcal{B}x})$,
and so by Lemma \ref{lem:SubmanifoldBdry} that $\mathbf{c}$ is not
in $\bdry\Delta_{\mathcal{B}x}$.
\end{proof}
We have now finished the proof of Theorem \ref{thm:CEDforCosetLattice}.
Let us review: Corollary \ref{cor:BasisForC(G)GivesCube} gave us
(ced-polytope), Proposition \ref{pro:CedBdryForC(G)} was (ced-bdry),
and Corollary \ref{cor:EarsCover_C(G)} gave us (ced-union). We notice
that the base-set with the earliest pattern is $\mathcal{B}_{0}=\{B_{i}^{0}\}$,
and that each $\Sigma_{\mathcal{B}_{0}x}$ is the face lattice of
a cube. Thus the first $\Delta_{\mathcal{B}x}$ is a polytope, while
all subsequent ones are proper subcomplexes of polytopes. Since we
proved in Corollary \ref{cor:CosetEarsAreShellable} that each $\Delta_{\mathcal{B}x}$
is shellable, we have (ced-topology).
\begin{note}
As previously mentioned, the convex ear decomposition we have constructed
is not a (dual) $EL$-ced. Although we would rather find an $EL$-ced
than a general convex ear decomposition, this is not in general possible
with the cubes we are looking at here. For example $\cosetlat(\mathbb{Z}_{2}^{2})$
has exactly three possible $\Sigma_{\mathcal{B}x}^{+}$'s, but the
homotopy type of the wedge of 6 $1$-spheres, so some $\vert\Sigma_{\mathcal{B}x}^{+}\vert\setminus\vert\Sigma_{\mathcal{B}'x'}^{+}\vert$
must be disconnected. The example of $\cosetlat(\mathbb{Z}_{2}^{2})$
is a geometric lattice, so does have an $EL$-ced (for a different
$EL$-labeling), but I have not been able to extend this to an $EL$-ced
for other relatively complemented groups. 

\medskip{}
The reader may have noticed that the constructed convex ear decomposition
is not far from being an $EL$-ced -- the difference is that each
$\Sigma_{\mathfrak{B}x}^{+}$ gives several {}``new'' ears -- and
that another possibility would be to extend the definition of $EL$-ced
to cover this case. However, as this would make the definition more
complicated, and as the gain seems relatively small, I have chosen
to leave the definition as presented.
\end{note}

\section{\label{sec:Poset-Products}Poset products}

Throughout this section, let $P_{1}$ and $P_{2}$ be bounded posets.

In Section \ref{sub:d-divisible-EL-labeling}, we defined the product
$P_{1}\times P_{2}$ and lower reduced product $P_{1}\lrtimes P_{2}$
of $P_{1}$ and $P_{2}$. It should come as no surprise that the \emph{upper
reduced product} $P_{1}\urtimes P_{2}$ of $P_{1}$ and $P_{2}$ is
defined as $\left((P_{1}\setminus\{\hat{1}\})\times(P_{2}\setminus\{\hat{1}\})\right)\cup\{\hat{1}\}$.
There is a natural inclusion of $P_{1}\lrtimes P_{2}$ (and of $P_{1}\urtimes P_{2}$)
into $P_{1}\times P_{2}$. 

Our goal in Section \ref{sec:Poset-Products} is to explain the background
and give proofs for Theorems \ref{thm:cedprod_intro} and \ref{thm:clcedprod_intro}.
The flavor and techniques of this section are different from the previous
two, so we pause to justify its connection with {}``Cubical Convex
Ear Decompositions''. Lower reduced products come up fundamentally
both in the $d$-divisible partition lattice, as we discussed in Section
\ref{sub:d-divisible-EL-labeling}, as well as in the coset lattice,
where $\cosetlat(G_{1}\times G_{2})\cong\cosetlat(G_{1})\lrtimes\cosetlat(G_{2})$
for groups $G_{1}$ and $G_{2}$ of co-prime orders. And some of the
decompositions in product posets are cubical after all: a cube is
the direct product of intervals, so if $C_{d}$ is the boundary of
the $d$-cube, with face lattice $L(C_{d})$, then $L(C_{d})=\lrprod_{1}^{d}L(C_{1})$.

\subsection{Poset products and polytopes}

I am told that the following proposition is folklore. It is also discussed
briefly in \cite{Kalai:1988}.
\begin{prop}
\label{pro:ProductOfPolytopes}If $\Sigma_{1}$ and $\Sigma_{2}$
are the face lattices of convex polytopes $X_{1}$ and $X_{2}$, then 
\begin{enumerate}
\item $\Sigma_{1}\times\Sigma_{2}$ is the face lattice of the {}``free
join'' $X_{1}\freejoin X_{2}$, a convex polytope.
\item $\Sigma_{1}\lrtimes\Sigma_{2}$ is the face lattice of the Cartesian
product $X_{1}\times X_{2}$, a convex polytope.
\item $\Sigma_{1}\urtimes\Sigma_{2}$ is the face lattice of the {}``free
sum'' of $X_{1}$ and $X_{2}$, a convex polytope.
\end{enumerate}
\end{prop}
Proposition \ref{pro:ProductOfPolytopes} guides us to a proof of
Lemma \ref{lem:posetpolytopeprod_intro}. Our main tool will be stellar
subdivision.

If $\Delta$ is a convex polytope with a proper face $\sigma$, then
a \emph{stellar subdivision} of $\Delta$ at $\sigma$, denoted $\stellarsd_{\sigma}\Delta$,
is $\conv\left(\Delta\cup\{v_{\sigma}\}\right)$, where $v_{\sigma}=w_{\sigma}-\varepsilon(w_{\Delta}-w_{\sigma})$
for some point $w_{\sigma}$ in the relative interior of $\sigma$,
some point $w_{\Delta}$ in the interior of $\Delta$, and a small
number $\varepsilon$. In plain language, we {}``cone off'' a new
vertex lying just over $\sigma$. Note that the relative interior
of a vertex is the vertex itself. Stellar subdivisions are discussed
in depth in \cite[III.2]{Ewald:1996} and \cite{Ewald/Shepard:1974}.

The main fact \cite[III.2.1, III.2.2]{Ewald:1996} that we will need
is that the faces of the boundary complex of $\stellarsd_{\sigma}\Delta$
are \[
\{\tau\,:\,\sigma\not\subseteq\tau\}\cup\left\{ v_{\sigma}*\tau\,:\,\tau\in\Delta\mbox{ with }\tau,\sigma\subseteq\tau'\mbox{ for some }\tau'\in\Delta,\mbox{ but }\sigma\not\subseteq\tau\right\} .\]
Thus the stellar subdivision replaces the faces containing $\sigma$
with finer subdivisions. 
\begin{example}
\cite[Section 2]{Ewald/Shepard:1974} The barycentric subdivision
of a polytopal $d$-complex $\Delta$ is the repeated stellar subdivision
of $\Delta$ along a reverse linear extension of its face lattice
$L(\Delta)$. That is, subdivide each $d$-dimensional face, then
each $(d-1)$-dimensional face, and so forth.
\end{example}
If $X$ is the boundary complex of a polytope, then let $\overline{X}$
denote $\conv X$, that is, the polytope of which $X$ is the boundary
complex.
\begin{lem}
\label{lem:LrPolytopeProduct}Suppose $P_{1}$ and $P_{2}$ are bounded
posets and that $\vert P_{1}\vert$ and $\vert P_{2}\vert$ are the
boundary complexes of polytopes. Then $\vert P_{1}\lrtimes P_{2}\vert$
can be obtained from the boundary complex of $\overline{\vert P_{1}\vert}\times\overline{\vert P_{2}\vert}$
by a sequence of stellar subdivisions.\end{lem}
\begin{proof}
Let $\Delta_{0}$ be the boundary complex of $\overline{\vert P_{1}\vert}\times\overline{\vert P_{2}\vert}$.
The faces of $\Delta_{0}$ are exactly the products $F^{(1)}\times F^{(2)}$,
where each $F^{(i)}$ is a non-empty face in $P_{i}$, and at least
one is proper. In particular the vertices are products of vertices
$v^{(1)}\times v^{(2)}$, where $v^{(i)}$ is in $P_{i}\setminus\{\hat{0},\hat{1}\}$.
We write this product of vertices as $(v^{(1)},v^{(2)})$, and think
of it as sitting in $\vert P_{1}\lrtimes P_{2}\vert$.

We start by ordering the elements $\{v^{(2)}\}$ of $P_{2}$ by a
reverse linear extension, and stellarly subdividing at each $\sigma=\overline{\vert P_{1}\vert}\times v^{(2)}$
in this order. Inductively assume that the faces containing $\sigma$
are those of the form $(\overline{\vert P_{1}\vert}\times F^{(2)})*C$,
where $F^{(2)}$ is a face of $\vert P_{2}\vert$ with top-ranked
vertex $v^{(2)}$, and $C$ is a simplex corresponding to (the simplicial
join of) a chain of elements of the form $(\hat{1},w^{(2)})$ (with
each $w^{(2)}>v^{(2)}$). Subdivision replaces these faces with those
of the form $(\overline{\vert P_{1}\vert}\times F_{0}^{(2)})*C*\{v_{\sigma}\}$,
where $F_{0}^{(2)}$ is a face having top-ranked vertex $<v^{(2)}$.
We abuse notation to call the newly introduced vertex $v_{\sigma}$
as $(\hat{1},v^{(2)})$, which puts us in the situation required to
continue our induction.

We next do the same procedure for the faces $v^{(1)}\times\overline{\vert P_{2}\vert}$.
That is, we order $\{v^{(1)}\}$ by a reverse linear extension of
$P_{1}$, and repeatedly perform stellar subdivision at each such
face according to this order. Since a face cannot contain both $\overline{\vert P_{1}\vert}$
and $\overline{\vert P_{2}\vert}$, these stellar subdivisions are
independent of the ones at $\overline{\vert P_{1}\vert}\times v^{(2)}$. 

After subdividing at all $\overline{\vert P_{1}\vert}\times v^{(2)}$
and $v^{(1)}\times\overline{\vert P_{2}\vert}$, we obtain a complex
$\Delta_{1}$. The vertex set of $\Delta_{1}$ is exactly $P_{1}\lrtimes P_{2}\setminus\{\hat{0},\hat{1}\}$.
The faces of $\Delta_{1}$ are $\{(F^{(1)}\times F^{(2)})*C\}$, where
$F^{(i)}$ is a face of $\vert P_{i}\vert$, and $C$ is a simplex
corresponding to either a chain of elements $(\hat{1},w^{(2)})$ or
a chain of elements $(w^{(1)},\hat{1})$.

Finally, we perform stellar subdivision at the vertices $v=(v^{(1)},v^{(2)})$,
where $v^{(i)}\in P_{i}\setminus\{\hat{0},\hat{1}\}$, in the order
of a reverse linear extension of $P_{1}\lrtimes P_{2}$. We make an
induction argument parallel to the one above: at the step associated
with vertex $v$, the faces containing $v$ are $\{(F^{(1)}\times F^{(2)})*C\}$.
As before, $F^{(i)}$ is a face of $\vert P_{i}\vert$ with top-ranked
vertex $v^{(i)}$, and $C$ corresponds to (the simplicial join of)
elements in a chain greater than $v$ in $P_{1}\lrtimes P_{2}$. Stellar
subdivision at $v$ replaces these faces with $\{(F_{0}^{(1)}\times F_{0}^{(2)})*C*\{v\}\}$,
where $F_{0}^{(i)}$ has greatest vertex $<v^{(i)}$, and we continue
the induction.

When we have subdivided at every vertex, we obtain a complex $\Delta_{2}$.
The faces of $\Delta_{2}$ are simply $\{C\}$, where $C$ is the
simplicial join of vertices in a chain of $P_{1}\lrtimes P_{2}$,
which is the definition of the order complex $\vert P_{1}\lrtimes P_{2}\vert$.\end{proof}
\begin{cor}
If If $P_{1}$ and $P_{2}$ are bounded posets such that $\vert P_{1}\vert$
and $\vert P_{2}\vert$ are the boundary complexes of polytopes, then
$\vert P_{1}\lrtimes P_{2}\vert$ and (by duality) $\vert P_{1}\urtimes P_{2}\vert$
are also boundary complexes of polytopes.
\end{cor}
For $P_{1}\times P_{2}$, a similar result holds. Recall that the
\emph{free join} $\Delta_{1}\freejoin\Delta_{2}$ of two polytopes
$\Delta_{1}$ and $\Delta_{2}$ is obtained by taking the convex hull
of embeddings of $\Delta_{1}$ and $\Delta_{2}$ into skew affine
subspaces of Euclidean space (of high enough dimension). The faces
of $\Delta_{1}\freejoin\Delta_{2}$, as hinted in Proposition \ref{pro:ProductOfPolytopes},
are $F^{(1)}\freejoin F^{(2)}$, and $\dim F^{(1)}\freejoin F^{(2)}=\dim F^{(1)}+\dim F^{(2)}+1$. 
\begin{lem}
Suppose $P_{1}$ and $P_{2}$ are bounded posets and that $\vert P_{1}\vert$
and $\vert P_{2}\vert$ are the boundary complexes of polytopes. Then
$\vert P_{1}\times P_{2}\vert$ can be obtained from the boundary
complex of $\overline{\vert P_{1}\vert}\freejoin\overline{\vert P_{2}\vert}$
by a sequence of stellar subdivisions.\end{lem}
\begin{proof}
Since the details of the proof are very similar to the preceding Lemma
\ref{lem:LrPolytopeProduct}, we provide a sketch only. Let $\Delta_{0}=\overline{\vert P_{1}\vert}\freejoin\overline{\vert P_{2}\vert}$.
Notice that the vertices of $\Delta_{0}$ are $\{\emptyset\freejoin v^{(2)}\}\cup\{v^{(1)}\freejoin\emptyset\}$,
while the edges are $\{v^{(1)}\freejoin v^{(2)}\}$. 

As in Lemma \ref{lem:LrPolytopeProduct}, we begin by ordering the
facets $\overline{\vert P_{1}\vert}\freejoin v^{(2)}$ and $v^{(1)}\freejoin\overline{\vert P_{2}\vert}$
according to reverse linear extensions of $P_{2}$ and $P_{1}$, and
inductively performing stellar subdivision. Each such subdivision
creates a vertex, which we name $(\hat{1},v^{(2)})$ or $(v^{(1)},\hat{1})$.
We obtain a complex $\Delta_{1}$ with faces $\{(F^{(1)}\freejoin F^{(2)})*C\}$
where $F^{(i)}$ is a proper face of $P_{i}$ (possibly empty), and
$C$ corresponds to a chain in the elements $\{(\hat{1},v^{(2)})\}$
or $\{(v^{(1)},\hat{1})$\}.

We then order the edges $v^{(1)}\freejoin v^{(2)}$ by a linear extension
of $P_{1}\times P_{2}$, and inductively perform stellar subdivision
to create vertices $(v^{(1)},v^{(2)})$. The resulting complex is
isomorphic to $\vert P_{1}\times P_{2}\vert$.\end{proof}
\begin{cor}
If $P_{1}$ and $P_{2}$ are bounded posets such that $\vert P_{1}\vert$
and $\vert P_{2}\vert$ are the boundary complexes of polytopes, then
$\vert P_{1}\times P_{2}\vert$ is also the boundary complex of a
polytope.
\end{cor}
This completes the proof of Lemma \ref{lem:posetpolytopeprod_intro}.

\subsection{\label{sub:CedsOfProductPosets}Convex ear decompositions of product
posets}

Let $P_{1}$ and $P_{2}$ be bounded posets with respective convex
ear decompositions $\{\Delta_{s}^{(1)}\}$ and $\{\Delta_{t}^{(2)}\}$.
Let $P$ be either $P_{1}\times P_{2}$, $P_{1}\lrtimes P_{2}$, or
$P_{1}\urtimes P_{2}$; with coordinate projection maps $p_{1}$ and
$p_{2}$. Take $d=\dim\vert P\vert$, $d_{1}=\dim\vert P_{1}\vert$,
and $d_{2}=\dim\vert P_{2}\vert$.

We define $\Delta_{s,t}$ to be the simplicial complex generated by
the maximal chains of $P$ that project to $\Delta_{s}^{(1)}$ in
the first coordinate, and $\Delta_{t}^{(2)}$ in the second. Order
these complexes lexicographically by $(s,t)$.
\begin{thm}
\label{thm:ced-respects-products}$\{\Delta_{s,t}\}$ is a convex
ear decomposition for $\vert P\vert$.\end{thm}
\begin{proof}
Lemma \ref{lem:posetpolytopeprod_intro} gives that $\Delta_{s,t}$
is a subcomplex of the boundary complex of a polytope, so (ced-polytope)
is satisfied.

The topology of various poset products is nicely discussed in Sundaram's
\cite[Section 2]{Sundaram:1994}. There are homeomorphisms \[
\vert P_{1}\lrtimes P_{2}\vert\approx\vert P_{1}\urtimes P_{2}\vert\approx\vert P_{1}\vert*\vert P_{2}\vert,\]
where $*$ is the join of topological spaces. This result goes back
to Quillen \cite[Proposition 1.9]{Quillen:1978}, although his notation
was much different -- Sundaram makes the connection in \cite[proof of Proposition 2.5]{Sundaram:1994}.
Walker \cite[Theorem 5.1 (d)]{Walker:1988} extends this to show that

\[
\vert P_{1}\times P_{2}\vert\approx\susp(\vert P_{1}\vert*\vert P_{2}\vert),\]
where $\susp$ denotes the topological suspension. Identical proofs
to Quillen's and Walker's show that $\Delta_{s,t}\approx\Delta_{s}*\Delta_{t}$
in the upper/lower reduced case, and that $\Delta_{s,t}\approx\susp(\Delta_{s}*\Delta_{t})$
in the direct product case. In particular, $\Delta_{s,t}$ is a $d$-ball
for $(s,t)>(1,1)$ and a $d$-sphere for $(s,t)=(1,1)$ by results
in PL-topology \cite[Proposition 2.23]{Rourke/Sanderson:1972}. We
have shown that (ced-topology) is satisfied.

It is clear that (ced-union) holds. It remains to check (ced-bdry).
\begin{claim}
$\bdry\Delta_{s,t}$ is exactly the set of all faces in $\Delta_{s,t}$
that project to either $\bdry\Delta_{s}^{(1)}$ or $\bdry\Delta_{t}^{(2)}$
(or both).\end{claim}
\begin{proof}
The boundary of a simplicial $d$-ball $\Delta$ is generated by the
$d-1$ faces that are contained in only a single facet of $\Delta$.
If $\mathbf{c}$ is a $d-1$ face of $\Delta_{s,t}$ (i.e., a chain
of length $d-1$), then at least one of $p_{1}(\mathbf{c})$ and $p_{2}(\mathbf{c})$
also has codimension 1.

Since $\Delta_{s,t}$ is defined to be the chains which project to
$\Delta_{s}^{(1)}$ and $\Delta_{t}^{(2)}$, we see that if $p_{1}(\mathbf{c})$
is $d_{1}-1$ dimensional and $\mathbf{c}$ is $d-1$ dimensional,
then $p_{1}(\mathbf{c})$ has exactly one extension in $\Delta_{s}^{(1)}$
if and only if $\mathbf{c}$ has exactly one extension in $\Delta_{s,t}$.
The argument if $p_{2}(\mathbf{c})$ has codimension 1 is entirely
similar.
\end{proof}
We now show both inclusions for (ced-bdry). If $\mathbf{d}$ is any
chain in $\bdry\Delta_{s,t}$ with $p_{1}(\mathbf{d})$ in $\bdry\Delta_{s}^{(1)}$,
then $p_{1}(\mathbf{d})$ is in $\Delta_{u}^{(1)}$ for some $u<s$
by (ced-bdry), so $\mathbf{d}$ is in $\Delta_{u,t}$; similarly if
$p_{1}(\mathbf{d})$ is maximal and $p_{2}(\mathbf{d})$ is in $\bdry\Delta_{t}^{(2)}$.
Thus $\bdry\Delta_{s,t}\subseteq\Delta_{s,t}\cap\left(\bigcup_{(u,v)<(s,t)}\Delta_{u,v}\right)$.

In the other direction: if $\mathbf{c}$ is in $\Delta_{s,t}$ and
$\Delta_{u,v}$ (for $(u,v)<(s,t)$), then $p_{1}(\mathbf{c})$ is
in both $\Delta_{s}^{(1)}$ and $\Delta_{u}^{(1)}$. If $s\neq u$,
then $p_{1}(\mathbf{c})$ is in $\bdry\Delta_{s}^{(1)}$, so $\mathbf{c}$
is in $\bdry\Delta_{s,t}$. A similar argument applies for $p_{2}$
when $s=u$. Thus, $\bdry\Delta_{s,t}\supseteq\Delta_{s,t}\cap\left(\bigcup_{(u,v)<(s,t)}\Delta_{u,v}\right)$,
and we have shown (ced-bdry), completing the proof.
\end{proof}

\subsection{\label{sub:ProductCLlabelings}Product $CL$-labelings}

In this subsection, we explicitly recall the product $CL$-labelings
introduced by Björner and Wachs in \cite[Section 10]{Bjorner/Wachs:1997},
and hinted at in Section \ref{sub:d-divisible-EL-labeling}. Since
there is no particular reason to work with dual labelings in Section
\ref{sec:Poset-Products}, I've chosen to work with standard (not
dual) $CL$-labelings, so that everything is {}``upside down'' relative
to Sections \ref{sec:d-divisible-Partition} and \ref{sec:coset-lattice}.
Since the root of an edge of the form $\hat{0}\lessdot x$ is always
$\emptyset$, we suppress the root from our notation in this case.

Let $P$ be a bounded poset with a $CL$-labeling $\lambda$ that
has label set $S_{\lambda}$. A label $s\in S_{\lambda}$ is \emph{atomic}
if it is used to label a cover relation $\hat{0}\lessdot x$ (for
any atom $x$), and \emph{non-atomic} if it is used to label any other
rooted cover relation. (In an arbitrary $CL$-labeling, a label can
be both atomic and non-atomic.) A $CL$-labeling is \emph{orderly}
if $S_{\lambda}$ is totally ordered and partitions into $S_{\lambda}^{-}<S_{\lambda}^{A}<S_{\lambda}^{+}$,
where every atomic label is in $S_{\lambda}^{A}$, and every non-atomic
label is either in $S_{\lambda}^{-}$ or $S_{\lambda}^{+}$. There
are similar definitions of \emph{co-atomic}, \emph{non-co-atomic},
and \emph{co-orderly}, and of course we can generalize to talk of
orderly and co-orderly chain edge labelings, even if the $CL$-property
is not met. 
\begin{lem}
\emph{(Björner and Wachs \cite[Lemma 10.18]{Bjorner/Wachs:1997})}\label{lem:OrderlyCLlabelingsExist}
Let $P$ be a bounded poset with a $CL$-labeling $\lambda$. Then
$P$ has an orderly $CL$-labeling $\lambda'$, and a co-orderly $CL$-labeling
$\lambda''$, such that any maximal chain $\mathbf{c}$ in $P$ has
the same set of ascents and descents under each of the three labelings
$\lambda$, $\lambda'$, and $\lambda''$.
\end{lem}
The proof involves constructing a recursive atom ordering from $\lambda$,
and then constructing a $CL$-labeling with the desired properties
from the recursive atom ordering. 
\begin{note}
The result of Lemma \ref{lem:OrderlyCLlabelingsExist} is not known
to be true if `$CL$' is replaced by `$EL$'.
\end{note}
\medskip{}
To find a $CL$-labeling of $P_{1}\times P_{2}$, we label each edge
in $P_{1}\times P_{2}$ with the edge in $P_{1}$ or $P_{2}$ to which
it projects. More formally, notice that any rooted cover relation
$(\mathbf{r},x\lessdot y)$ projects to a cover relationship in one
coordinate, and to a point in the other. Then the \emph{product labeling},
denoted $\lambda_{1}\times\lambda_{2}$, labels $(\mathbf{r},x\lessdot y)$
with $\lambda_{i}\left(p_{i}(\mathbf{r}),p_{i}(x\lessdot y)\right)$,
where $i$ is the coordinate where projection is nontrivial. It is
straightforward to show that $\lambda_{1}\times\lambda_{2}$ is a
$CL$-labeling if $\lambda_{1}$ and $\lambda_{2}$ are $CL$-labelings
of $P_{1}$ and $P_{2}$, and where we order $S_{\lambda_{1}}\cup S_{\lambda_{2}}$
by any shuffle of $S_{\lambda_{1}}$ and $S_{\lambda_{2}}$ \cite[Proposition 10.15]{Bjorner/Wachs:1997}.

The idea behind finding a $CL$-labeling of $P_{1}\lrtimes P_{2}$
(or similarly $P_{1}\urtimes P_{2}$) is to restrict $\lambda_{1}\times\lambda_{2}$
to $P_{1}\lrtimes P_{2}$. For a cover relation $x\lessdot y$ where
$x\neq\hat{0}$, this works very well, as $x\lessdot y$ in $P_{1}\lrtimes P_{2}$
is also a cover relation in $P_{1}\times P_{2}$, and the roots project
straightforwardly. The problem comes at cover relations $\hat{0}\lessdot y$,
which project to a cover relation in both $P_{1}$ and $P_{2}$. Here,
we need to combine the labels $\lambda_{1}\left(\hat{0}\lessdot p_{1}(y)\right)$
and $\lambda_{2}\left(\hat{0}\lessdot p_{2}(y)\right)$. 

The orderly labelings constructed in Lemma \ref{lem:OrderlyCLlabelingsExist}
are a tool to perform this combination in a manner that preserves
the $CL$-property. Let $P_{1}$ and $P_{2}$ have orderly $CL$-labelings
$\lambda_{1}$ and $\lambda_{2}$, with disjoint label sets $S_{1}$
and $S_{2}$. Suppose the label sets are shuffled together as \[
S_{1}^{-}<S_{2}^{-}<S_{1}^{A}<S_{2}^{A}<S_{1}^{+}<S_{2}^{+}.\]
Then the \emph{lower reduced product labeling} $\lambda_{1}\lrtimes\lambda_{2}$
labels an edge $\hat{0}\lessdot y$ with the word $\lambda_{1}\left(p_{1}(\hat{0}\lessdot y)\right)\lambda_{2}\left(p_{2}(\hat{0}\lessdot y)\right)$
in $S_{1}^{A}S_{2}^{A}$ (lexicographically ordered), while all other
rooted edges $(\mathbf{r},x\lessdot y)$ (for $x\neq\hat{0})$ are
labeled with the nontrivial projection $\lambda_{i}\left(p_{i}(\mathbf{r}),p_{i}(x\lessdot y)\right)$
as in $\lambda_{1}\times\lambda_{2}$. Björner and Wachs proved \cite[Theorems 10.2 and 10.17]{Bjorner/Wachs:1997}
that $\lambda_{1}\lrtimes\lambda_{2}$ is a $CL$-labeling of $P_{1}\lrtimes P_{2}$. 

Similarly, if $\lambda_{1}$ and $\lambda_{2}$ are co-orderly $CL$-labelings
of $P_{1}$ and $P_{2}$, with disjoint label sets shuffled together
as for the orderly labelings above, we define the \emph{upper reduced
product labeling} $\lambda_{1}\urtimes\lambda_{2}$ as follows. Label
an edge of the form $(\mathbf{r},x\lessdot\hat{1})$ with the word
\[
\lambda_{1}\left(p_{1}(\mathbf{r}),p_{1}(x\lessdot\hat{1})\right)\lambda_{2}\left(p_{2}(\mathbf{r}),p_{2}(x\lessdot\hat{1})\right)\]
 in $S_{1}^{A}S_{2}^{A}$, and all other edges $(\mathbf{r},x\lessdot y)$
(for $y\neq\hat{1})$ as in $\lambda_{1}\times\lambda_{2}$. Then
\cite[Theorems 10.2 and 10.17]{Bjorner/Wachs:1997} gives us that
$\lambda_{1}\urtimes\lambda_{2}$ is a $CL$-labeling of $P_{1}\urtimes P_{2}$.
\begin{example}
\label{exa:DDivLabelingAsProduct}The labeling $\lambda_{\mbox{div}}$
we constructed for the $d$-divisible partition lattice was an co-orderly
$EL$-labeling of the dual lattice: actually, $S^{A}$ was just $\{0\}$.
As discussed in Lemma \ref{lem:DDivProductStructure}, intervals split
as products, and the restriction of $\lambda_{\mbox{div}}$ to an
interval splits as the appropriate product labeling.
\end{example}
We summarize in the following theorem:
\begin{thm}
\emph{(Björner and Wachs \cite[Proposition 10.15 and Theorem 10.17]{Bjorner/Wachs:1997})}
Let $P_{1}$ and $P_{2}$ be posets, with respective labelings $\lambda_{1}$
and $\lambda_{2}$.
\begin{enumerate}
\item If $\lambda_{1}$ and $\lambda_{2}$ are $CL$-labelings ($EL$-labelings),
then $\lambda_{1}\times\lambda_{2}$ is a $CL$-labeling ($EL$-labeling)
of $P_{1}\times P_{2}$.
\item If $\lambda_{1}$ and $\lambda_{2}$ are orderly $CL$-labelings,
then $\lambda_{1}\lrtimes\lambda_{2}$ is a $CL$-labeling of $P_{1}\lrtimes P_{2}$.
\item If $\lambda_{1}$ and $\lambda_{2}$ are co-orderly $CL$-labelings,
then $\lambda_{1}\urtimes\lambda_{2}$ is a $CL$-labeling of $P_{1}\urtimes P_{2}$.
\end{enumerate}
\end{thm}

\subsection{$CL$-ceds of product posets}

Fix our notation as in Section \ref{sub:CedsOfProductPosets}, but
suppose in addition that $P_{1}$ and $P_{2}$ have $CL$-ceds $\{\Sigma_{s}^{(1)}\}$
and $\{\Sigma_{t}^{(2)}\}$ with respect to the $CL$-labelings $\lambda_{1}$
and $\lambda_{2}$. Denote the resulting ears of new chains as $\{\Delta_{s}^{(1)}\}$
and $\{\Delta_{t}^{(2)}\}$, as in Section \ref{sub:ELcedsAndCLceds}.
Then take $\Sigma_{s,t}$ to be the appropriate product of $\Sigma_{s}^{(1)}$
and $\Sigma_{t}^{(2)}$, and $\Delta_{s,t}$ to be the associated
ear of new chains. 

We first notice that there is no inconsistency with the notation used
in Section \ref{sub:CedsOfProductPosets}:
\begin{lem}
A maximal chain $\mathbf{c}$ is in $\Delta_{s,t}$ if and only if
$p_{1}(\mathbf{c})$ is in $\Delta_{s}^{(1)}$ and $p_{2}(\mathbf{c})$
is in $\Delta_{t}^{(2)}$.\end{lem}
\begin{proof}
The statement follows straightforwardly from the fact that the maximal
chains of $\Sigma_{s,t}$ are those that project to $\Sigma_{s}^{(1)}$
and $\Sigma_{t}^{(2)}$. 
\end{proof}
As we did in Section \ref{sub:CedsOfProductPosets}, order the $\{\Sigma_{s,t}\}$
according to the lexicographic order of $(s,t)$. Let $\lambda$ be
the appropriate product $CL$-labeling, where we assume without loss
of generality via Lemma \ref{lem:OrderlyCLlabelingsExist} that $\lambda_{1}$
and $\lambda_{2}$ are orderly or co-orderly. Then we will prove:
\begin{thm}
$\{\Sigma_{s,t}\}$ is a $CL$-ced for $\vert P\vert$ with respect
to $\lambda$.\end{thm}
\begin{proof}
Proposition \ref{pro:ProductOfPolytopes} tells us that (CLced-polytope)
is satisfied, and (CLced-union) is immediate from the definitions. 

For (CLced-bdry), we work backwards, and notice that we have already
shown in Theorem \ref{thm:ced-respects-products} that $\bdry\Delta_{s,t}=\Delta_{s,t}\cap\left(\bigcup_{u,v<s,t}\Delta_{u,v}\right)$.
Lemma \ref{lem:SubmanifoldBdry} meanwhile gives that $\bdry\Delta_{s,t}=\Delta_{s,t}\cap\overline{\vert\Sigma_{s,t}\vert\setminus\Delta_{s,t}}$,
hence that a chain $\mathbf{c}$ with extensions in both $\Delta_{s,t}$
and $\Delta_{u,v}$ has an extension in $\mathcal{M}(\Sigma_{s,t})\setminus\mathcal{M}(\Delta_{s,t})$,
as required. (A direct proof is also straightforward.)

It remains to check (CLced-desc). Although the statement of this property
is very similar to \cite[Theorem 10.17]{Bjorner/Wachs:1997} (which
says that $\lambda$ is a $CL$-labeling), the proof in \cite{Bjorner/Wachs:1997}
uses some machinery. So we work from scratch, as follows.

If $P=P_{1}\times P_{2}$ and we are considering the rooted interval
$[x,y]_{\mathbf{r}}$, then the labels of a maximal chain $\mathbf{c}_{\mathbf{r}}$
on $[x,y]_{\mathbf{r}}$ are the same as the labels of $p_{1}(\mathbf{c})_{p_{1}(\mathbf{r})}$
union with the labels of $p_{2}(\mathbf{c})_{p_{2}(\mathbf{r})}$,
{}``shuffled together'' in some order. Thus, if $\mathbf{c_{r}}$
is descending, then the projections must also be descending. Since
the label sets $S_{1}$ and $S_{2}$ are taken to be disjoint, there
is a unique way of shuffling the two label sets (and so the two chains)
together to get a descending chain.

For $P=P_{1}\lrtimes P_{2}$, the proof is the same unless $x=\hat{0}$.
In this case, the first label of a maximal chain $\mathbf{c}$ is
in $S_{1}^{A}S_{2}^{A}$, while the first label of the projections
are in $S_{1}^{A}$ and $S_{2}^{A}$, respectively. If $\mathbf{c}$
is descending, then all labels after the first are from $S_{1}^{-}$
or $S_{2}^{-}$, since $S_{1}^{-}<S_{2}^{-}<S_{1}^{A}S_{2}^{A}<S_{1}^{+}<S_{2}^{+}$,
and thus $p_{1}(\mathbf{c})$ and $p_{2}(\mathbf{c})$ are descending,
and we argue as before.

The proof for $P=P_{1}\urtimes P_{2}$ is entirely similar to that
for $P_{1}\lrtimes P_{2}$.
\end{proof}

\section{\label{sec:Further-Questions}Further questions}

The close relationship between the techniques used in Sections \ref{sec:d-divisible-Partition}
and \ref{sec:coset-lattice} leads us to ask the following question.

\begin{question}Are there other families of posets with similar structure
to $\Pi_{n}^{d}$ and $\cosetlat(G)$? Can the techniques used in
Sections \ref{sec:d-divisible-Partition} and \ref{sec:coset-lattice}
be used to construct dual $EL$-labelings and $EL$-ceds?\end{question}

What we mean by `similar' here is not clear. At the least, we need
a poset $P$ where every interval of the form $[a,\hat{1}]$ is supersolvable,
and where the supersolvable structure is canonically determined, i.e.,
such that we can label all edges of $P\setminus\{\hat{0}\}$ in a
way that restricts to a supersolvable labeling on each such $[a,\hat{1}]$
interval. We then need a way to sign the edges giving an $EL$-labeling,
and the poset has to somehow be `wide' or `rich' enough to have an
$EL$-ced.

One possible source of such examples is the theory of exponential
structures. An \emph{exponential structure} is a family of posets
with each upper interval isomorphic to the partition lattice, and
each lower interval isomorphic to a product of smaller elements in
the same family. Exponential structures were introduced in \cite{Stanley:1978},
where the family of $d$-divisible partition lattices was shown to
be one example. Shellings are constructed for some other examples
in \cite{Sagan:1986,Welker:1995}.

\begin{question}Can techniques like those used in Section \ref{sec:d-divisible-Partition}
(and Section \ref{sec:coset-lattice}) be used to construct dual $EL$-labelings
and/or $EL$-ceds of exponential structures besides the $d$-divisible
partition lattice?\end{question}

However, it is not a priori clear how to construct a labeling that
restrict to a supersolvable labeling on any $[a,\hat{1}]$ for exponential
structures. In examples even finding an $EL$-labeling often seems
to be a difficult problem.

A question suggested by the results of Section \ref{sec:Poset-Products}
is:

\begin{question}Are there other operations on posets that preserve
convex ear decompositions and/or $CL$-ceds?\end{question}

For example, Schweig shows \cite[Theorem 5.1]{Schweig:2006} that
rank-selected supersolvable and geometric lattices have convex ear
decompositions. Do all rank-selected subposets of posets with convex
ear decompositions have a convex ear decomposition? Are there any
other useful constructions that preserve having a convex ear decomposition
and/or $EL$-ced? A place to start looking would be in Björner and
Wach's papers \cite{Bjorner/Wachs:1983,Bjorner/Wachs:1996,Bjorner/Wachs:1997},
where they answer many such questions for $EL$/$CL$-labelings.

\section*{Acknowledgements}

Thanks to Ed Swartz for introducing me to convex ear decompositions;
and to him, Jay Schweig, and my graduate school advisor Ken Brown
for many helpful discussions about them. Tom Rishel listened to and
commented on many intermediate versions of the results and definitions
of this paper. Sam Hsiao helped me in understanding the material of
Proposition \ref{pro:ProductOfPolytopes}, and in looking for its
extension to Lemma \ref{lem:posetpolytopeprod_intro}. Vic Reiner
pointed out that the labeling based on pivots of $\cosetlat(G)$ that
I used in \cite{Woodroofe:2007} was really a supersolvable labeling,
which suggested the improved $EL$-labeling used in Section \ref{sec:coset-lattice}.
Volkmar Welker suggested exponential structures as a possible area
for further exploration. The anonymous referee gave many helpful comments.

\bibliographystyle{hamsplain}
\bibliography{5_Users_paranoia_Documents_Research_Master}

\end{document}